\documentclass{article}
\usepackage[utf8]{inputenc}
\usepackage{amsmath,amssymb,amsthm,mathrsfs,color,times,textcomp,sectsty,verbatim}
\usepackage{xcolor}
\usepackage[colorlinks=true]{hyperref}
\hypersetup{urlcolor=blue, citecolor=red, linkcolor=blue}
\usepackage[colorinlistoftodos]{todonotes}
\usepackage[toc,page]{appendix}
\usepackage{enumitem, hyperref}
\usepackage{tikz-cd}
\usepackage{hyperref}
\hypersetup{
    colorlinks,
    citecolor=red,
    filecolor=orange,
    linkcolor=blue,
    urlcolor=purple
  }

\usepackage[a4paper, total={6in, 8in}]{geometry}

\newcommand{\isom}{\cong} 
\newcommand{\union}{\cup}
\newcommand{\ints}{\cap}

\DeclareMathOperator{\End}{End}

\newcommand{\SO}{\mathrm{SO}}

\newcommand{\Spin}{\mathrm{Spin}}

\newcommand{\CC}{\mathbb{C}}

\newcommand{\OO}{\mathbb{O}}

\newcommand{\RR}{\mathbb{R}}

\newcommand{\ZZ}{\mathbb{Z}}

\newcommand{\cA}{\mathcal{A}}
\newcommand{\cB}{\mathcal{B}}

\newcommand{\cK}{\mathcal{K}}

\newcommand{\cN}{\mathcal{N}}

\newcommand{\del}{\partial}

\newcommand{\vol}{\mathrm{Vol}}

\newcommand{\id}{\mathrm{Id}}
\newcommand{\grad}{\mathrm{grad}}
\newcommand{\hess}{\mathrm{Hess}}
\newcommand{\tr}{\mathrm{tr}}
\newcommand{\Cay}{\mathcal{C}ay}

\newcommand{\FFJ}{\mathfrak{J}}

\newcommand{\FFX}{\mathfrak{X}}
\newcommand{\FFY}{\mathfrak{Y}}

\newcommand{\matii}[2]{\left(\begin{array}{cc} #1\\#2\end{array}\right)}
\newcommand{\matiii}[3]{\left(\begin{array}{ccc} #1\\#2\\#3\end{array}\right)}
\newcommand{\mativ}[4]{\left(\begin{array}{cccc} #1\\#2\\#3\\#4\end{array}\right)}

\newcommand{\Addresses}{{
  \bigskip

\textsc{Department of Mathematics, University of Michigan, Ann Arbor, MI, 48109.} \par\nopagebreak 

\textit{E-mail address}: \texttt{\email{ruanyp@umich.edu}}.

}}

\newcommand*{\email}[1]{%
    \normalsize\href{mailto:#1}{#1}\par
    }

\newtheorem{thm}{Theorem}[section]

\newtheorem{prop}[thm]{Proposition}
\newtheorem{lem}[thm]{Lemma}
\newtheorem{cor}[thm]{Corollary}

\newtheorem{fact}[thm]{Fact}
\theoremstyle{definition}
\newtheorem{dfn}[thm]{Definition}

\newtheorem*{rmk}{Remark}

\def\namedlabel#1#2{\begingroup
    #2%
    \def\@currentlabel{#2}%
    \phantomsection\label{#1}\endgroup
}
\makeatother

\title{The Cayley hyperbolic space and volume entropy rigidity}
\author{Yuping Ruan}
\date{}

\begin{document}

\maketitle
\begin{abstract}
Let $M$ be a Riemannian manifold with dimension greater or equal to $3$ which admits a complete, finite-volume Riemannian metric $g_0$ locally isometric to a rank-1 symmetric space of non-compact type. The volume entropy rigidity theorem \cite[Theor\'em\`e principal]{BCG1} asserts that $g_0$ minimizes a normalized volume growth entropy among all complete, finite-volume, Riemannian metric on $M$. We will repair a gap in the proof when $g_0$ is locally isometric to the Cayley hyperbolic space.
\end{abstract}
\tableofcontents
\section{Introduction}\label{s0}
Let $h(g)$ be the \emph{volume growth entropy} of a closed $n$-dimensional Riemannian manifold $(M,g)$, i.e.
$$h(g)=\lim_{r\to \infty}\frac{1}{r}\ln(\vol (B(x,r))),$$
where $B(x,r)$ denotes the ball of radius $r$ centered at some point $x$ in the universal cover of $(M,g)$. When $M$ is not compact but has finite volume, we define $h(g)$ instead by
$$h(g)=\inf\left\{ s\geq 0\left|\exists K>0\mathrm{~s.t.~}\int_{\widetilde M}e^{-s d(x,z)}dg(z)<K\right.\right\},$$ 
where $dg$ denotes the volume form on the universal cover $\widetilde M$ of $M$ induced by the metric $g$. We introduce the \emph{normalized entropy} 
$$\mathrm{ent}(g)=h(g)^n\vol(M,g)$$
so that it remains invariant under scaling. In \cite{BCG1} Besson-Courtois-Gallot proved the following theorem.
\begin{thm}[Besson-Courtois-Gallot, Theor\'em\`e principal,\cite{BCG1}]\label{main_thm}

Let $M$ be a Riemannian manifold of dimension $n$ which admits a complete, finite-volume Riemannian metric $g_0$ locally isometric to a rank 1 symmetric space of non-compact type with dimension at least $3$. Let $g$ be another complete, finite-volume Riemannian metric on $M$. Assuming that $(\widetilde M,g)$, the universal cover of $(M,g)$, has bounded geometry (see the remark below), then
\begin{align}\label{ent_rigid_ineq}
\mathrm{ent}(g)\geq\mathrm{ent}(g_0)
\end{align}
with equality achieved if and only if the map $\id:(M,g)\to (M,g_0)$ is homotopic to an isometry.

More generally, if $(N,g)$ is any $n$-dimensional finite-volume Riemannian manifold of bounded geometry, and if $f:N\to M$ is any proper, smooth map with non-zero degree, then
$$\mathrm{ent}(N,g)\geq|\deg f|\mathrm{ent}(M,g_0)$$
with equality achieved if and only if $f$ is homotopic to a Riemannian covering.
\end{thm}
\begin{rmk}
\begin{enumerate}
\item[(1).] A Riemannian manifold $X$ has \emph{bounded geometry} if for any $\epsilon>0$, there exists some constant $C>0$ such that any geodesic ball of radius $\epsilon$ has volume greater or equal to $C$;
\item[(2).] The above theorem is slightly different from \cite[Theor\`em\'e principal]{BCG1} because Besson-Courtois-Gallot assumed that $M$ and $N$ are compact. The non-compact case requires a few extra arguments and was proved by Boland-Connell-Souto in \cite[Theorem 1.3]{BCS}. Also, P. Storm removed the bounded geometry condition in \cite[Theorem 1.1]{Storm}.
\end{enumerate}
\end{rmk}
We noticed a technical gap in the proof of the above theorem when $(M,g_0)$ is locally isometric to the Cayley hyperbolic space $\OO\mathbf{H}^2$. The gap appears in the assertion (see \cite[page 751]{BCG1}) which claims that there exist fiber-wise linear maps $J_k:T\widetilde M\to T\widetilde M$, $k=1,...,7$ such that the Hessian of the \emph{Busemann function} satisfies the following formula
\begin{align*}&\hess B_\theta|_x(v,w) \\
=&g_0(v,w)-dB_\theta(v)dB_\theta(w)+\sum_{k=1}^7dB_\theta(J_kv)dB_\theta(J_kw), ~~x\in M, v,w\in T_x\widetilde M,
\end{align*}
where $\widetilde M$ denotes the universal cover of $(M,g_0)$ and $B_\theta$ is the Busemann function with respect to the geodesic ray starting at a fixed point towards a point $\theta$ in the visual boundary $\del_\infty \widetilde M$ of $\widetilde M$.
We observe that such maps $J_k$ \textbf{do not} exist in the Cayley hyperbolic case due to non-associativity of octonionic multiplication, which makes this case different from other rank 1 symmetric spaces of non-compact type. The goal of this paper is to repair this gap so that Theorem \ref{main_thm} still holds in the Cayley hyperbolic space.

In Section \ref{s1}, we will introduce the ``matrix model'' and the ``vector model'' for the Cayley hyperbolic space. The ``vector model'' has appeared in multiple resources (e.g. \cite{Parker}). It is an analogue of the Minkowski model for the real hyperbolic space. The ``matrix model'' is a variant of the model in \cite[\S19. Spaces of $\RR$-rank 1]{Mostow73} which is convenient for us to define the distance function. Towards the end of Section \ref{s1}, we will introduce Cayley lines as analogues of complex lines and explain their relations with curvature data in the ``vector model'' for the Cayley hyperbolic space. These computations lead to the gap in \cite[page 751]{BCG1}, i.e. the aforementioned $J_k$ maps do not exist. This gap is related to some kind of incompatibility between the sectional curvature data and multiplying vectors by purely imaginary octonions due to non-associativity of octonionic multiplication. To be specific, we compare the complex hyperbolic space $\CC\mathbf{H}^n$ with the Cayley hyperbolic space $\OO\mathbf{H}^2$ (assuming their sectional curvatures take values in $[-4,-1]$ ). For any unit vector $v\in T\CC\mathbf{H}^n$, the sectional curvature between $v$ and $iv$ is always $-4$ because they lie on the same complex line. However, for any unit vector $v\in T\OO\mathbf{H}^2$ and any purely imaginary unit octonion $e$, the sectional curvature between $v$ and $ev$ can vary because they are not necessarily on the same Cayley line. Nevertheless, the computations in this section shows that there are still a lot of similarities between the Cayley structure on $\OO\mathbf{H}^2$ and the complex structure on $\CC\mathbf{H}^n$. 

Section \ref{s2} is an outline of Besson-Courtois-Gallot's original proof for Theorem \ref{main_thm} and Subsection \ref{s3} points out the gap in their proof. This gap is located in the technical part of Besson-Courtois-Gallot's argument regarding the following inequality
\begin{align*}\frac{\det\left(\int_{\del_\infty\widetilde M}(dB_\theta)^2(x)d\mu(\theta)\right)^{1/2}}{\det\left(\int_{\del_\infty\widetilde M}\hess B_\theta(x)d\mu(\theta)\right)}\leq\left(\frac{\sqrt{16}}{16+8-2}\right)^{16}.\end{align*}
Here $\del_\infty \widetilde M$ is the visual boundary (diffeomorphic to a sphere of dimension $\dim M-1$) of $\widetilde M$ and $\mu$ is a probability measure in the Lebesgue measure class of $\del_\infty \widetilde M$. We will repair this gap in Subsections \ref{s4}-\ref{s6}. 

The main idea to repair the gap is to control the ``bad outcome'' caused by the loss of the aforementioned $J_k$ maps. Subsection \ref{s4} reduces the above inequality to a pure linear algebra problem. To be specific, we first notice that we can define $J_k$ maps on each Cayley line (but not the whole tangent space!) which behave nicely with respect to curvature. This enables us to find a suitable basis on $T_x\widetilde M$ such that symmetric bilinear forms in both the numerator and the denominator have matrices almost as nice as the case in the complex (or quaternionic) hyperbolic setting. Therefore we can translate the geometric information into properties on matrices and use these properties instead for the rest of the proof. In Subsection \ref{s5}, we prove Lemma \ref{key_lemma} which deals with the linear algebra problem from Subsection \ref{s4}. This is the main part and the most techincal part of this paper. We first observe that we can make small adjustments to Besson-Courtois-Gallot's proof in \cite[Appendice B]{BCG1} so that the inequality holds when eigenvalues of these matrices are relatively close to their averages respectively. This corresponds to \hyperlink{case1}{\textbf{Case 1}} in Lemma \ref{key_lemma} considered as the ``good'' case. The complementary case is much more unpleasant to work with and hence considered as the ``bad'' case (see \hyperlink{case2}{\textbf{Case 2}} in Lemma \ref{key_lemma}). In this case, some eigenvalues of the matrices are extremely large compared to other eigenvalues. Although the proof in this case is very technical, the overall idea behind the proof is much simpler. By \cite[B.4 Lemme]{BCG1}, the extreme eigenvalues will actually help us in the proof. Therefore the general idea in this case is to show that these extreme eigenvalues help us more than the possible trouble created by the loss of the aforementioned $J_k$ maps. Subsection \ref{s6} proves a slightly stronger inequality which is also used in Besson-Courtois-Gallot's proof (see \cite[B.5 Proposition]{BCG1}).

\textbf{Acknowledgements}: I would like to heartily thank my advisor Ralf Spatzier for his support during the entire work. I sincerely thank G\'erald Besson, Gilles Courtois and Sylvain Gallot for patiently proofreading this paper and providing valuable suggestions. I am also very grateful to Chris Connell for helpful and thorough discussions on this subject.

\section{Octonions and the Cayley hyperbolic space}\label{s1}
This section is a very brief review on octonions and the Cayley hyperbolic space. A more detailed version can be found in \cite[Appendix A]{Ruan1}.

The set of octonions $\OO$ is an $8$-dimensional non-associative, non-commutative division algebra over $\RR$. Let $\langle\cdot,\cdot\rangle$ be the Euclidean inner product on $\OO$ and $|\cdot|$ the induced norm. Then we have the following properties. (See \cite{Springer00})
\begin{enumerate}
\item[(1).] $|ab|=|a||b|$ for any $a,b\in\OO$;
\item[(2).] $\langle ab,ac\rangle=\langle ba,ca\rangle=|a|^2\langle b,c\rangle$ for any $a,b,c\in \OO$;
\item[(3).] $\langle ac,bd\rangle+\langle ad,bc\rangle=2\langle a,b\rangle \langle c,d\rangle$ for any $a,b,c,d\in\OO$;
\item[(4).] $\overline{a}=2\langle a,1\rangle-a$ for any $a\in \OO$;
\item[(5).] $2\langle a,b\rangle=2\langle \overline{a},\overline{b}\rangle=\overline{a}b+\overline{b}a=a\overline{b}+b\overline{a}$ for any $a,b\in\OO$;
\item[(6).] $(ba)\overline{a}=\overline{a}(ab)=|a|^2b$ for any $a,b\in \OO$;
\item[(7).] $a(\overline{b}c)+b(\overline{a}c)=(c\overline{a})b+(c\overline{b})a=\langle a,b\rangle c$ for any $a,b,c\in\OO$;
\item[(8).] (Moufang Identities)
\begin{enumerate}
\item[(i).] $(ab)(ca)=a((bc)a)$ for any $a,b,c\in\OO$;
\item[(ii).] $a(b(ac))=(a(ba))c$ for any $a,b,c\in\OO$;
\item[(iii).] $b(a(ca))=((ba)c)a$ for any $a,b,c\in\OO$;
\end{enumerate}
\item[(9).] Multiplications involving only two octonions are associative.
\end{enumerate}
Let 
$$I_{1,2}=\matiii{1&0&0}{0&-1&0}{0&0&-1}$$
and 
$$\FFJ(1,2,\OO)=\left\{\FFX\in\mathrm{Mat}_{3\times 3}(\OO):I_{1,2}\FFX^*I_{1,2}=\FFX\right\}.$$
Any element $\FFX$ can be written in the following form
$$\FFX(\theta,a)=\matiii{\theta_1&a_3&\overline{a}_2}{-\overline{a}_3&-\theta_2&-a_1}{-a_2& -\overline{a}_1&-\theta_3},\quad \theta_j\in\RR, a_j\in\OO, j=1,2,3,$$
where $\theta=(\theta_1,\theta_2,\theta_3)$ and $x=(x_1,x_2,x_3)$.
\begin{dfn}[``Matrix model'']\label{matrix_model_OH^2}
We define the Cayley hyperbolic space $\OO\mathbf{H}^2$ as
$$\OO\mathbf{H}^2=\left\{\FFX\in\FFJ(1,2,\OO):\FFX^2=\FFX,\tr(\FFX)=1,\FFX_{11}>0 \right\}.$$
\end{dfn}
\begin{rmk}
The way we define the Cayley hyperbolic space using matrices is slightly different from that in \cite[\S19 Spaces of $\RR$-rank 1]{Mostow73}. One can eventually prove that they are equivalent. See \cite[Proposition A.4]{Ruan1} for a detailed argument.
\end{rmk}
\begin{prop}\label{model_equiv_OH^2}
For any trace $1$ idempotent $\FFX\in\FFJ(1,2,\OO)$ with $\FFX_{11}\neq 0$, there exists a unique vector $(\theta,a,b)\in\RR_+\times\OO^2$ such that 
$$\FFX=\mathrm{sgn}(\FFX_{11})I_{1,2}(\theta,b,c)^*(\theta,b,c),$$
where $\mathrm{sgn}(t)=t/|t|$ when $t\neq 0$. The set
$$\FFJ_{1,0}:=\{\FFX\in\FFJ(1,2,\OO):\FFX^2=\FFX,\tr(\FFX)=1,\FFX_{11}=0\}$$
is isomorphic to $\OO\mathbf{P}^1\isom S^8$.
\end{prop}
\begin{proof}
See \cite[Proposition A.2]{Ruan1}.
\end{proof}

Therefore we have the following alternative definition for the Cayley hyperbolic space (also see \cite[page 87]{Parker}).
\begin{dfn}[``Vector model'']\label{vector_model_OH^2}
The Cayley hyperbolic space can be alternatively defined as $$\OO\mathbf{H}^2=\{(\theta,a,b)\in\RR_{+}\times\OO^2: \theta^2-|a|^2-|b|^2=1\}.$$
\end{dfn}
\begin{rmk}
One can easily check that for any $\FFX\in\OO\mathbf{H}^2$, $\FFX_{11}\geq 1$.

\end{rmk}

We refer to \cite[\S19. Spaces of $\RR$-rank 1]{Mostow73}, \cite[Appendix A]{Ruan1} and \cite{Baez} for an overview of related subjects. A very detailed and general theory can be found in \cite{Springer1}, \cite{Springer00} by T. A. Springer and F. D. Veldkemp for further reference.

Denote by $g_0$ the symmetric metric on $\OO\mathbf{H}^2$ and $M=(\OO\mathbf{H}^2,g_0)$ such that the distance function on $\OO\mathbf{H}^2$ is given by
$$\cosh(2d(\FFX,\FFY))=\tr(\FFX\FFY+\FFY\FFX)-1,\quad\forall \FFX,\FFY\in\OO\mathbf{H}^2$$
as in \cite{Mostow73}. $(\FFX\FFY+\FFY\FFX)/2=:\FFX\circ\FFY$ is also known as the Jordan multiplication of $\FFX$ and $\FFY$. We identify $T_{x_0}M_0$ with $\OO^2$ such that each unit vector $v=(a,b)\in\OO^2$ corresponds to the initial vector of $\gamma_v(t):=(\cosh(t),a\sinh(t),b\sinh(t))$. 
We will compute the Riemannian curvature data at $x_0$ by understanding the \textit{geodesic hinge} $\angle\FFX x_0\FFY$ for any $\FFX,\FFY\in M\setminus \{x_0\}$, where $\angle\FFX x_0\FFY$ consists of two geodesic segments $\overline{\FFX x_0}$, $\overline{\FFY x_0}$ and the angle $\measuredangle\FFX x_0\FFY$. A \textit{comparison hinge} of $\angle\FFX x_0\FFY$ in some space form $M'$ is a geodesic hinge $\angle\FFX 'x_0'\FFY'$ in $M'$ with the same angle such that the lengths of geodesic segments $\overline{\FFX x_0}$, $\overline{\FFY x_0}$ and $\overline{\FFX' x_0'}$, $\overline{\FFY' x_0'}$ are equal respectively.

\begin{prop}
The following hold for the Cayley hyperbolic space.
\begin{enumerate}\label{curv_OH^2}
\item[(1).] For any unit vector $v=(a,b)\in\OO^2$, $\gamma_v(t)$ gives a unit speed geodesic starting at $x_0$ with initial vector $v$; 
\item[(2).] The Riemannian metric at $x_0$ is given by the Euclidean inner product on $\OO^2$. Hence the map $\chi:\OO^2\to M$ such that
$$\chi(a,b)=\left(\cosh(|(a,b)|),\frac{a}{|(a,b)|}\sinh(|(a,b)|),\frac{b}{|(a,b)|}\sinh(|(a,b)|)\right)$$
gives the geodesic normal coordinates centered at $x_0$ (hence $d\chi:\OO^2\to T_{x_0}M_0$ is the isometric correspondence from $\OO^2$ with Euclidean inner product to $T_{x_0}M_0$ mentioned above)
\item[(3).] For any $v=(a,b)\in\OO^2$, denote by 
\begin{align*}
\Cay(v)=\begin{cases}
\displaystyle \OO\cdot(1,a^{-1}b),\quad& a\neq0; \\
\displaystyle \OO\cdot(0,1),&a=0
\end{cases}
\end{align*}
the \emph{Cayley line} containing $v$. Then for any non-parallel pair of non-zero vectors $(a,b),(c,d)$ contained in the same Cayley line and any pair of points $\FFX\in\gamma_{(a,b)}(\RR_+)$ and $\FFY\in\gamma_{(c,d)}(\RR_+)$, the comparison hinge $\angle \FFX' x_0'\FFY'$ of $\angle \FFX x_0\FFY$ in a space form of constant sectional curvature $-4$ satisfies $d(\FFX',\FFY')=d(\FFX,\FFY)$;
\item[(4).] For any $a,b\in\OO\setminus\{0\}$, any pair of points $\FFX\in\gamma_{(a,0)}(\RR_+)$ and $\FFY\in\gamma_{(0,b)}(\RR_+)$, the comparison hinge $\angle \FFX' x_0'\FFY'$ of $\angle \FFX x_0\FFY$ in a space form of constant sectional curvature $-1$ satisfies $d(\FFX',\FFY')=d(\FFX,\FFY)$.
\end{enumerate}
\end{prop}
\begin{proof}
The proof can be found in \cite[Proposition A.5]{Ruan1}. We will present the proof for reader's convenience.
\begin{enumerate}
\item[(1).] This follows easily by direct computations.
\item[(2).] Let $(a,b),(c,d)$ be unit vectors in $\OO^2$. Hence the inner product of these two vectors is given by
\begin{align*}
&-\left.\frac{d}{dt}\right|_{t=0}d(\gamma_{(a,b)}(t),\gamma_{(c,d)}(1)) \\=&-\frac{\left.\frac{d}{dt}\right|_{t=0}\left(\tr[\gamma_{(a,b)}(t)\gamma_{(c,d)}(1)+\gamma_{(c,d)}(1)\gamma_{(a,b)}(t)]-1\right)}{2\sinh(2)} \\
=&\frac{2\sinh(1)\cosh(1)(\langle a,c\rangle+\langle b,d\rangle)}{2\sinh(2)} 
=\langle a,c\rangle+\langle b,d\rangle.
\end{align*}
\item[(3).] Let $a,b,c$ be unit octonions and $\theta\in[0,\pi/2)$. Let $v=(\cos\theta,a\sin\theta)$. For any $t_1,t_2>0$, we have
\begin{align*}
&\cosh(2d(\gamma_{bv}(t_1),\gamma_{cv}(t_2)))\\
=&\tr[\gamma_{bv}(t_1)\gamma_{cv}(t_2)+\gamma_{cv}(t_2)\gamma_{bv}(t_1)]-1 \\
=&2\cosh^2(t_1)\cosh^2(t_2)-4\cosh(t_1)\sinh(t_1)\cosh(t_2)\sinh(t_2)\langle bv,cv\rangle -1\\
&+2\sinh^2(t_1)\sinh^2(t_2)(\cos^4\theta+\sin^4\theta)+4\sinh^2(t_1)\sinh^2(t_2)\cos^2\theta\sin^2\theta \\
=&2\cosh^2(t_1)\cosh^2(t_2)+2\sinh^2(t_1)\sinh^2(t_2)-\sinh(2t_1)\sinh(2t_1)\langle bv,cv\rangle -1 \\
=&\cosh(2t_1)\cosh(2t_2)-\sinh(2t_1)\sinh(2t_1)\langle bv,cv\rangle.
\end{align*}
Notice that $t_1=d(\gamma_{bv}(t_1),E_1)$ and $t_2=d(\gamma_{cv}(t_2),E_1)$, the above equation coincides with the law of cosines in a space form with constant sectional curvature $-4$.
\item[(4).] Let $a,b$ be unit octonions and $t_1,t_2>0$. Write $v=(a,0)$ and $w=(0,b)$. Then
\begin{align*}
\cosh(2d(\gamma_{v}(t_1),\gamma_{w}(t_2)))=2\cosh^2(t_1)\cosh^2(t_2)-1,
\end{align*}
which implies that 
$$\cosh(d(\gamma_{v}(t_1),\gamma_{w}(t_2)))=\cosh(t_1)\cosh(t_2).$$
The above equation coincides with the law of cosines in a space form with constant sectional curvature $-1$.\qedhere
\end{enumerate}
\end{proof}
A direct corollary of the above proposition is the following.
\begin{cor}\label{curv_data_OH^2}
For any non-zero $v,w\in \OO^2$ the following hold.
\begin{enumerate}
\item[(1).] If $v,w$ belong to the same Cayley line and $v\not\in\RR w$, then the sectional curvature of the $2$-dimensional plane spanned by $d\chi(v),d\chi(w)$ is $-4$;
\item[(2).] If $\Cay(v)\perp\Cay(w)$, then the sectional curvature of the $2$-dimensional plane spanned by $d\chi(v),d\chi(w)$ is $-1$.
\item[(3).] Recall that from classical results that $\OO\mathbf{H}^2=F_4^{-20}/\Spin(9)$ with $\Spin(9)$ the stabilizer of $x_0$. Since $F_4^{-20}$ acts by isometries on $M=(\OO\mathbf{H}^2,g_0)$, identifying $T_{x_0}M_0$ with $\OO^2$, one can view $\Spin(9)$ as a subgroup of $\SO(\OO^2)\isom\SO(16)$. In particular, $\Spin(9)$ maps Cayley lines to Cayley lines and acts transitively on the set of all Cayley lines.  
\end{enumerate}
\end{cor}
\begin{rmk}
Let $p$ be a fixed point in $\CC\mathbf{H}^n$. Recall that for any unit vector $v\in T^1_p\CC\mathbf{H}^n$ with respect to the symmetric metric, there exists a linear map $J\in\End(T_p\CC\mathbf{H}^n)$ such that the sectional curvature between $v$ and $J(v)$ is $-4$. The same argument holds true in quaternionic hyperbolic spaces but \textbf{NOT} in the Cayley hyperbolic space. This is due to the non-associativity of octonionic multiplication. In fact, if such a map $J$ exists, without loss of generality we can assume that $p=x_0$ and identify the tangent space at $x_0$ with $\OO^2$ via $d\chi$ as in Proposition \ref{curv_OH^2}. If $J(1,0)=(a,0)$ for some unit octonion $a\in\OO\setminus \RR$, then for any unit octonion $b$ and any $\theta\in[0,\pi/2)$, $J(\cos\theta,b\sin\theta)=a(\cos\theta,b\sin\theta)$. In particular, $J(0,1)=(0,a)$. Similarly we have $J(b\sin\theta,\cos\theta)=a(b\sin\theta,\cos\theta)$ for any unit octonion $b$ and any $\theta\in[0,\pi/2)$, which implies that $J(v)=av$ for any $v\in\OO^2$. Therefore, for any octonions $b\neq 0,c$, $J(b,bc)=(ab,a(bc))\in\Cay(b,bc)=\Cay(1,c)$, which implies that $(ab)c=a(bc)$. Notice that if $(ab)c=a(bc)$ for any $b,c\in\OO$, then $a$ must be real. This contradicts the assumption that $a\in\OO\setminus\RR$.
\end{rmk}

Let $F_4^{-20}=\cK\cA\cN$ be the Iwasawa decomposition of $F_4^{-20}$. Denote by $v_{l,t}=(\delta_{1t}e_t,\delta_{2t}e_t)$ with $l=1,2$, $\delta_{lm}$ the Kronecker delta, $0\leq t\leq 7$ and $\{e_t\}_{0\leq t\leq7}$ the standard orthonormal basis for $\OO$. Then we can construct $\cA\cN$-invariant vector fields $\widetilde\xi_{l,t}$ such that  
$$\widetilde\xi_{l,t}(x_0)=\Psi(v_{l,t})\quad l=1,2\mathrm{~and~}0\leq t\leq 7,$$
where $\Psi$ is the same as $d\chi$ in Proposition \ref{curv_OH^2}. Define $\Psi_{x}:\OO^2\to T_xM_0$ such that
$$\Psi_{x}(v_{l,t})=\widetilde\xi_{l,t}(x),\quad l=1,2\mathrm{~and~}0\leq t\leq7.$$
For simplicity we denote by
$$\Cay(\xi)=\Psi_{x}\left(\Cay(\Psi_{x}^{-1}(\xi))\right),\quad0\neq\xi\in T_xM$$
the \emph{Cayley line} containing $\xi$. Since the isometry group acts transitively on the Cayley hyperbolic space, the same arguments as in Corollary \ref{curv_data_OH^2} can be applied to all points in $M_0$. Namely,
\begin{cor}\label{curv_data_general}
For any $x\in M_0$ and any non-zero $\xi,\xi'\in T_{x}M$.
\begin{enumerate}
\item[(1).] If $\xi,\xi'$ belong to the same Cayley line and $\xi\not\in\RR \xi'$, then the sectional curvature $K_{M}(\xi,\xi')=-4$;
\item[(2).] If $\Cay(\xi)\perp\Cay(\xi')$, then the sectional curvature $K_{M}(\xi,\xi')=-1$.
\end{enumerate}
\end{cor}

\section{An overview of Besson-Courtois-Gallot's approach to volume entropy rigidity in rank one symmetric spaces of non-compact type}\label{s2}
Denote by $(\widetilde N,g)$ and $(\widetilde M,g_0)$ the universal (Riemannian) covers of $(N,g)$ and $(M,g_0)$ respectively. Let $\Gamma$ be the fundamental group of $M$. For any $x\in \widetilde M$, denote by $\mu_x$ the Patterson-Sullivan measure (See \cite{Al}) at $x$ for $\Gamma$. In this particular case, $\mu_x$ is also the probability measure on the topological boundary $\del_\infty \widetilde M$ invariant under $\mathrm{Stab}_x(\mathrm{Isom}(\widetilde M))$. One can view $\mu_x$ as the push forward of the Haar measure on $T^1_x\widetilde M$ to the topological boundary via opposite geodesic rays centered at $x$. For any $s>h(g)$ we define a family of probability measures $\{\sigma_y^s\}_{y\in\widetilde N}$ on $\del_\infty \widetilde M$ as
$$\sigma_y^s(U)=\frac{\int_{\widetilde N}e^{-sd_g(y,z)}\mu_{\widetilde f(z)}(U)dg(z)}{\int_{\widetilde N}e^{-sd_g(y,z')}dg(z')},\quad~\forall y\in\widetilde N, U\subset\del_\infty \widetilde M,$$
where $\widetilde f:\widetilde N\to\widetilde M$ is a lift of $f$ to universal covers with respect to a pair of base points $p\in N$ and $f(p)\in M$.

Fix a point $x_0\in\widetilde M$, we define the \emph{Busemann function} as
$$B_{\theta}(x)=b_\theta(x,x_0)=\lim_{t\to\infty}d(x,\gamma_\theta(t))-t,\quad\forall x\in M, \theta\in\del_\infty M.$$
where $\gamma_\theta(t)$ is the geodesic ray starting from $x_0$ towards $\theta$. Then we define a $C^1$-map (due to implicit function theorem. See \cite[pages 635-636]{BCG2}.) 
$\widetilde F_s:\widetilde N\to\widetilde M$ such that $\widetilde F_s(y)$ is the unique critical point of the following function
$$\cB_{s,y}(x):=\int_{\del_\infty\widetilde M}B_\theta(x)d\sigma_y^s(\theta),$$
where uniqueness comes from the convexity of Busemann functions and $\sigma_y^s$ is in the measure class of the Haar measure. It is not hard to see that the definition here is independent of the choice of $x_0$ up to an additive constant and the map $\tilde F_s$ naturally descends to a map $F_s:N\to M$. Also when $s$ goes to $\infty$, $F_s$ goes to $f$, which implies that $F_s$ is homotopic to $f$. Direct computation shows that for any $v\in T_{\widetilde F_s(y)}\widetilde M$ and any $w\in T_y\widetilde N$, we have
\begin{align*}
0=&D_yD_{x=\widetilde F_s(y)}\cB_{s,y}(x)(v,w)=\int_{\del_\infty\widetilde M}\hess B_\theta|_{\widetilde F_s(y)}(D_y\widetilde F_s(w),v)d\sigma_y^s(\theta) \\ 
&-s\int_{\widetilde N}\int_{\del_\infty M}dB_\theta|_{\widetilde F_s(y)}(v)\langle \nabla_yd_g(y,z),w\rangle  \frac{e^{-sd_g(y,z)}}{\int_{\widetilde N}e^{-sd_g(y,z)}dg(z)}d\mu_{\widetilde f(z)}(\theta)dg(z).
\end{align*}
Hence as matrices
\begin{align}\label{DF_s}
D_y\widetilde F_s=sW_y^{-1}\cdot H_y,
\end{align}
where $W_y$ and $H_y$ are matrices defined by bilinear forms on $T_{\widetilde F_s(y)}\widetilde M\times T_{\widetilde F_s(y)}\widetilde M$ and $T_y\widetilde N\times T_{\widetilde F_s(y)}\widetilde M$ respectively such that
$$W_y(v',v)=\int_{\del_\infty\widetilde M}\hess B_\theta|_{\widetilde F_s(y)}(v',v)d\sigma_y^s(\theta)$$
and 
$$H_y(w,v)=\int_{\widetilde N}\int_{\del_\infty M}dB_\theta|_{\widetilde F_s(y)}(v)\langle \nabla_yd_g(y,z),w\rangle  \frac{e^{-sd_g(y,z)}}{\int_{\widetilde N}e^{-sd_g(y,z')}dg(z')}d\mu_{\widetilde f(z)}(\theta)dg(z)$$
for any $w\in T_y\widetilde N$ and $v,v'\in T_{\widetilde F_s(y)}\widetilde M$.

In particular, by H\"older's inequality we have 
\begin{align}\label{trace_control}
|\tr H_y|\leq&\tr\left(\int_{\del_\infty\widetilde M}dB^2_\theta|_{\widetilde F_s(y)}d\sigma_y^s(\theta)\right)^{1/2}\tr\left(\int_{\widetilde N}\frac{\langle \nabla_yd_g(y,z),w\rangle^2e^{-sd_g(y,z)}}{\int_{\widetilde N}e^{-sd_g(y,z')}dg(z')} dg(z)\right)^{1/2}\leq 1
\end{align}
and
\begin{align}\label{Holder}
|\det H_y|\leq&\det\left(\int_{\del_\infty\widetilde M}dB^2_\theta|_{\widetilde F_s(y)}d\sigma_y^s(\theta)\right)^{1/2}\det\left(\int_{\widetilde N}\frac{\langle \nabla_yd_g(y,z),w\rangle^2e^{-sd_g(y,z)}}{\int_{\widetilde N}e^{-sd_g(y,z')}dg(z')} dg(z)\right)^{1/2}\nonumber\\
\leq&\left(\frac{1}{\sqrt{n}}\right)^n\det\left(\int_{\del_\infty\widetilde M}dB^2_\theta|_{\widetilde F_s(y)}d\sigma_y^s(\theta)\right)^{1/2}.
\end{align}
Therefore
\begin{align}\label{det_est}
|\mathrm{Jac} \widetilde F_s(y)|\leq \left(\frac{s}{\sqrt{n}}\right)^n\frac{\det\left(\int_{\del_\infty\widetilde M}dB^2_\theta|_{\widetilde F_s(y)}d\sigma_y^s(\theta)\right)^{1/2}}{\det\left(\int_{\del_\infty\widetilde M}\hess B_\theta|_{\widetilde F_s(y)}d\sigma_y^s(\theta)\right)}.
\end{align}
In order to prove the inequality \eqref{ent_rigid_ineq}, it suffices to show that the right hand side of \eqref{det_est} is smaller than or equal to $(s/h(g_0))^n$. Note that in the non-compact case, one also needs to verify that $F_s$ is proper. Moreover
\begin{align}\label{sharp_det_est}
\left(\frac{s}{\sqrt{n}}\right)^n\frac{\det\left(\int_{\del_\infty\widetilde M}dB^2_\theta|_{\widetilde F_s(y)}d\sigma_y^s(\theta)\right)^{1/2}}{\det\left(\int_{\del_\infty\widetilde M}\hess B_\theta|_{\widetilde F_s(y)}d\sigma_y^s(\theta)\right)}\leq\left(\frac{s}{h(g_0)}\right)^n\left[1-K\sum_{j=1}^n\left(\nu_j-\frac{1}{n}\right)^2\right]
\end{align}
for some positive constant $K$, where $\nu_1,...,\nu_n$ are eigenvalues of $\int_{\del_\infty\widetilde M}dB^2_\theta|_{\widetilde F_s(y)}d\sigma_y^s(\theta)$. This is also the part where the proof in \cite{BCG1} has a gap. We will explain this in Subsection \ref{s3}. 

When equality holds in \eqref{ent_rigid_ineq}, assuming $h(g)=h(g_0)$ after scaling, \cite[7.6 Lemme]{BCG1} and \cite[7.8 Lemme]{BCG1} imply that there exists a sequence $s_k\to h(g)^+$ such that $F_{s_k}$ converge to a 1-Lipschitz map $F:N\to M$ uniformly on compact subsets. Note that in the non-compact case, $F$ is also a proper map following \cite[Proposition 5.1]{BCS} or \cite[Lemma 6.7]{CF0}. Then it follows from \cite[C.1 Proposition]{BCG1} that $F$ is a local isometry. Since we can still prove \eqref{sharp_det_est} in the Cayley hyperbolic setting (see Corollary \ref{key_cor}), the whole argument for the equality case in \eqref{ent_rigid_ineq} still works for the Cayley hyperbolic case without any changes. 

\section{The gap in the proof of Theorem \ref{main_thm} from \cite{BCG1} and its repair}
\subsection{The gap in the proof of Theorem \ref{main_thm} from \cite{BCG1} regarding the Cayley hyperbolic space}\label{s3}
Following the notations from the previous sections, we recall that the \emph{Busemann function} satisfies
$$B_{\theta}(x)=b_\theta(x,x_0)=\lim_{t\to\infty}d(x,\gamma_\theta(t))-t,\quad\forall x\in \widetilde M, \theta\in\del_\infty \widetilde M.$$
where $\gamma_\theta(t)$ is the geodesic ray starting from $x_0$ towards $s$. In \cite[page 36]{CF} we have
\begin{align}\label{hess_gen_formula}\hess B_\theta(x)=\sqrt{-R(\grad B_\theta,\cdot,\grad B_\theta,\cdot)},\end{align}
where $R$ denotes the Riemannian curvature tensor.

In \cite[page 751]{BCG1}, Besson-Courtois-Gallot claimed that there exist fiberwise linear maps $J_l:T\widetilde M\to T\widetilde M$ with $1\leq l\leq d-1$ such that 
\begin{align}\label{hess_easy_case}\hess B_\theta(x)=g_0-dB_\theta^2+\sum_{l=1}^{d-1}(dB_\theta\circ J_l)^2.\end{align}

By Corollary \ref{curv_data_general}, $J_l(\grad B_\theta(x))$ is contained in $\Cay(\grad B_\theta(x))$. However, the remark after Corollary \ref{curv_data_OH^2} proves that such linear maps do not exist. Hence we believe there is a gap here. Since this calculation only affects \eqref{sharp_det_est}, Theorem \ref{main_thm} still works in the Cayley hyperbolic space following \cite{BCG1} provided we can prove \eqref{sharp_det_est}.

\subsection{Repairing the gap: reducing a weaker version of \eqref{sharp_det_est} into a linear algebra problem}\label{s4}
Following the notations from Section \ref{s3}, we will first discuss a weaker version of \eqref{sharp_det_est} when $M$ is locally isometric to the Cayley hyperbolic space. Namely the following
\begin{align}\label{coarse_det_est}
\left(\frac{s}{\sqrt{n}}\right)^n\frac{\det\left(\int_{\del_\infty\widetilde M}dB^2_\theta|_{\widetilde F_s(y)}d\sigma_y^s(\theta)\right)^{1/2}}{\det\left(\int_{\del_\infty\widetilde M}\hess B_\theta|_{\widetilde F_s(y)}d\sigma_y^s(\theta)\right)}\leq\left(\frac{s}{h(g_0)}\right)^n.
\end{align}
We will leave the proof of \eqref{sharp_det_est} in Subsection \ref{s6}. Notice that $n=16$, $h(g_0)=16+8-2=22$ in the Cayley hyperbolic setting, \eqref{coarse_det_est} follows from the following proposition.
\begin{prop}\label{ineq_prop}
Let $x\in M$ be an arbitrary point. Denote by $\mu$ a probability measure on $\del_\infty \widetilde M$ in the Lebesgue measure class. For any positive function $\rho\in L^2(\del_\infty \widetilde M,d\mu)$, we have
\begin{align}\label{VIP_ineq}\frac{\det\left(\int_{\del_\infty \widetilde M}(dB_\theta)^2(x)d\mu(\theta)\right)^{1/2}}{\det\left(\int_{\del_\infty \widetilde M}\hess B_\theta(x)d\mu(\theta)\right)}\leq\left(\frac{\sqrt{16}}{16+8-2}\right)^{16}.\end{align}
\end{prop}
We want to reduce Proposition \ref{ineq_prop} to a pure linear algebra problem. To this end, we first introduce some reductions.
\begin{proof}[Reductions of Proposition \ref{ineq_prop}]
For any $w\in T_x\widetilde M\setminus\{0\}$, we denote by $\pi_{\Cay(w)}:T_x\widetilde M\to T_x\widetilde M$ the orthogonal projection onto $\Cay(w)$. For simplicity we write $\pi_{x,\theta}=\pi_{\Cay(\grad B_\theta(x))}$. In the Cayley hyperbolic space, by \eqref{hess_gen_formula} we can explicitly compute the Hessian of Busemann functions as
$$\hess B_\theta(x)(\cdot,\cdot)=g_0(\cdot,\cdot)-2dB_\theta(\cdot)dB_\theta(\cdot)+g_0(\cdot,\pi_{x,\theta}(\cdot)).$$
We will be using this formula in place of \eqref{hess_easy_case} in the Cayley hyperbolic setting because \eqref{hess_easy_case} does not hold in the Cayley hyperbolic case (see Subsection \ref{s3}). We also note that a corresponding version of this formula is equivalent to \eqref{hess_easy_case} in the complex (or quaternionic) hyperbolic setting. 

To simplify notations, we let
$$H=\int_{\del_\infty M}(dB_\theta)^2(x)d\mu(\theta)$$
and 
$$\widehat{H}=\int_{\del_\infty M}\langle\cdot,\pi_{x,\theta}(\cdot)\rangle d\mu(\theta),$$
where $\langle\cdot,\cdot\rangle$ denotes the inner product induced by $g_0$. If we identify $H$ and $\widehat H$ with their induced self-adjoint linear operators, $H$ is the $\mu$-weighted average of orthogonal projections onto 1-dimensional subspaces spanned by unit vectors and $\widehat H$ is  the $\mu$-weighted average of orthogonal projections onto Cayley lines of unit vectors. Then the inequality \eqref{VIP_ineq} is equivalent to
\begin{align}\label{simplification}
\frac{(\det H)^{1/2}}{\det \left(\id-H+(\widehat{H}-H)\right)}\leq\left(\frac{\sqrt{n}}{n+d-2}\right)^{n}= \left(\frac{\sqrt{16}}{16+8-2}\right)^{16}
\end{align}
where $n=\dim M$ and $d=\dim_\RR\OO=8$. 

\textbf{A remark on \eqref{simplification} in the complex (or quaternionic) hyperbolic setting:} One can similarly define $H$ and $\widehat H$ in the complex (or quaternionic) hyperbolic setting as  the $\mu$-weighted averages of orthogonal projections onto unit vectors and their complex (or quaternionic) lines respectively. In those cases which are already well-understood in \cite[Appendice B]{BCG1}, $\widehat H=\sum_{l=0}^{d-1}J_lHJ_l^{-1}$, where $J_0=\id$ and $J_1,...,J_{d-1}$, $d=2,4$ come from the complex (or quaternionic) structure(s). Here we want to focus on the following two properties of $\widehat H$ and $H$ in the complex (or quaternionic) hyperbolic cases:
\begin{enumerate}
\item[\hypertarget{P1}{(P1).}] Let $v,w$ be unit vectors. If $v,w$ lie in the same complex (or quaternionic) line, then $w^T\widehat H w= v^T\widehat H v$;
\item[\hypertarget{P2}{(P2).}] $\widehat H-H$ is the sum of $d-1$ matrices which are all conjugate to $H$ via orthogonal matrices. Moreover, these orthogonal matrices are given by the complex (or quaternionic) structure on $M$. In particular, these orthogonal matrices map every complex (or quaternionic) line to itself.
\end{enumerate} 

\textbf{Going back to the Cayley hyperbolic case.} Unfortunately, due to the lack of $J_l$ maps compatible with curvature data (see the remark after Corollary \ref{curv_data_OH^2}), the second property described above does not hold for the Cayley hyperbolic setting. Nonetheless, \hyperlink{P1}{(P1)} and a weaker version of \hyperlink{P2}{(P2)} still hold for the Cayley hyperbolic space. To be specific, we have the following lemma.

\begin{lem}\label{octonion_struct}
Let $v$ be any unit vector in $T^1_xM$. We denote by $\widehat{H}|_{\Cay(v)}$ the symmetric bilinear form on $\Cay(v)$ induced by $\widehat H$. Then we have the following properties of $\widehat H$ and $H$.
\begin{enumerate}
\item[\hypertarget{PP1}{(P1').}] $\widehat{H}|_{\Cay(v)}=\lambda_v\cdot g_0|_{\Cay(v)}$, where $\lambda_v$ is a real number related to $v$. In particular, if $v$ is an eigenvector of $\widehat H$, all non-zero vectors in $\Cay(v)$ are eigenvectors of $\widehat H$ with the same eigenvalue. This is the Cayley hyperbolic version of \hyperlink{P1}{(P1)}.
\item[\hypertarget{PP2}{(P2').}] There exists orthogonal linear maps $J_{t,v}\in \mathrm{O}(\Cay(v))$, $t=0,1,...,7$ such that $J_{0,v}=\id$ and 
$$\widehat{H}|_{\Cay(v)}=\sum_{t=0}^7 H|_{\Cay(v)}\circ (J_{t,v}(\cdot),J_{t,v}(\cdot)).$$
This is weaker than an analogue of \hyperlink{P2}{(P2)} in the Cayley hyperbolic setting.
\end{enumerate}
\end{lem}
\begin{proof}[Proof of Lemma \ref{octonion_struct}]

\begin{enumerate}
\item[(P1').] Since $\widehat{H}$ is a weighted average of 2-forms $\langle\cdot,\pi_{\Cay(w)}(\cdot)\rangle$, where $w\in T^1_x M$, it suffices to prove the conclusion for $\langle\cdot,\pi_{\Cay(w)}(\cdot)\rangle$. It is natural to consider the octonionic coordinates $\Psi_x:\OO^2\to T_xM$ introduced at the end of Section \ref{s1}. Since $\mathrm{Stab}_x(\mathrm{Isom}(M_0))$ acts transitively on $T_x^1M_0$, we can assume without loss of generality that $v=\Psi_{x}(1,0)$ and $w=\Psi_{x}(b\cos\theta, ba\sin\theta)$, where $\theta\in\RR$ and $a,b$ are unit octonions. Then for any unit octonion $c$ and any $v_c=\Psi_{x}(c,0)$, we have
$$\pi_{\Cay(w)}(v_c)=\Psi_{x}(c\cos^2\theta,ca\sin\theta\cos\theta).$$
Hence
\begin{align*}
\langle v_c,\pi_{\Cay(w))}(v_c)\rangle=\langle (c,0),(c\cos^2\theta,ca\sin\theta\cos\theta)\rangle=\cos^2\theta,
\end{align*}
The first assertion follows since the above inner product is independent of $c$.
\item[(P2').] Let $w\in T^1_xM_0$. Similar to the proof of the first assertion, we can assume without loss of generality that $v=\Psi_{x}(1,0)$ and $w=\Psi_{x}(b\cos\theta, ba\sin\theta)$, where $\theta\in\RR$ and $a,b$ are unit octonions. Also it suffices to prove the conclusion for $\langle\cdot,\pi_{\Cay(w)}(\cdot)\rangle$.
Let $e_0=1,e_1,...,e_7$ be the standard orthonormal basis for $\OO$. Then 
$$\{w_t:=\Psi_{x}(\overline{e}_tb\cos\theta,(\overline{e}_tb)a\sin\theta):0\leq t\leq 7\}$$ 
is an orthonormal basis for $\Cay(w)$. Let $\pi_u:T_xM\to T_xM$ be the orthogonal projection onto $\RR u$ for any non-zero vector $u\in T_xM$. Then
$$\langle v',\pi_{\Cay(w)}(v')\rangle=\sum_{t=0}^7\langle v',\pi_{w_t}v'\rangle, \forall v'\in\Cay(v).$$
Notice that any $v'$ has the form $\Psi_{x}(c,0)$ for some unit octonion $c$ and
\begin{align*}
&\langle v',\pi_{w_t}v'\rangle\\
=&\langle\Psi_{x}(c,0),\Psi_{x}(\overline{e}_tb\cos\theta,(\overline{e}_tb)a\sin\theta)\rangle^2 \\
=&\langle c,\overline{e}_tb\cos\theta\rangle^2 \\
=&\langle e_tc,b\cos\theta\rangle^2 \\
=&\langle \Psi_{x}(e_tc,0),w\rangle^2 
=\langle \Psi_{x}(e_t\Psi_{x}^{-1}(v')),\pi_w(\Psi_{x}(e_t\Psi_{x}^{-1}(v')))\rangle,\quad 0\leq t\leq 7.
\end{align*}
Choose $J_{t,v}(\cdot)=\Psi_{x}(e_t\Psi_{x}^{-1}(\cdot))$ and the second assertion follows.\qedhere
\end{enumerate}
\end{proof}

Back to the proof of Proposition \ref{ineq_prop}, \hyperlink{P1'}{(P1')} suggests that $\widehat{H}$ has at most 2 eigenvalues $\lambda,\mu\geq 0$ with $\lambda+\mu=1$. When $\lambda\neq\mu$, the two eigenspaces are a pair of Cayley lines perpendicular to each other. This is because the largest and the smallest eigenvalue of $\widehat{H}$ can be defined as
$$\lambda_{\max}=\max_{v\in T^1_xM}\widehat{H}(v,v)\mathrm{~and~}\lambda_{\min}=\min_{v\in T^1_xM}\widehat{H}(v,v).$$ 
A unit vector $v\in T^1_xM$ is an eigenvector  of $\widehat H$ corresponding to the largest (smallest resp.) eigenvalue if and only if $\widehat H(v,v)$ equals to the largest (smallest resp.) eigenvalue. Hence \hyperlink{P1'}{(P1')} implies that the eigenspace of $\widehat H$ corresponding to the largest (smallest resp.) eigenvalue must contain all Cayley lines which have non-trivial intersections with it.
Therefore under a suitable orthonormal basis $H$ and $\widehat{H}$ have matrices of the form
$$H=\matii{A&C}{C^*&B}\mathrm{~~and~~}\widehat{H}=\matii{\lambda\cdot\id&0}{0&\mu\cdot\id},$$
where $A, B$ are $8\times 8$ diagonal matrices with trace equal to $\lambda,\mu$ respectively. The trace data of $A$ and $B$ follows from \hyperlink{P2'}{(P2')}.  In particular, 
\begin{align}\label{matrix_formulation}H=\matii{A&C}{C^*&B}\mathrm{~and~}\widehat H-H=\matii{\lambda\cdot\id-A&-C}{-C^*&\mu\cdot\id-B}\end{align}
are positive semi-definite.

Before we outline the main idea of the proof for \eqref{simplification} in the Cayley hyperbolic setting, we first recall Besson-Courtois-Gallot's method in proving the complex (or quaternionic) hyperbolic version of \eqref{simplification}. (See \cite[B.3 Lemme]{BCG1}.) 

\textbf{\hypertarget{BCG's proof sketch}{Besson-Courtois-Gallot's proof} for \eqref{simplification} in the complex (or quaternionic) hyperbolic setting: }We first apply the log-concavity of the determinant function for positive semi-definite matrices (see \cite[B.2 Lemme]{BCG1}) and obtain
\begin{align}\label{old_trick}
\det(\id-H+(\widehat{H}-H))\geq (n+d-2)^{n}\det\left(\frac{\id-H}{n-1}\right)^{\frac{n-1}{n+d-2}}\det\left(\frac{\widehat{H}-H}{d-1}\right)^{\frac{d-1}{n+d-2}},
\end{align}
where $n$ is the dimension of a complex (or quaternionic) hyperbolic space and $d=2$ (or $4$). It follows from \cite[B.4 Lemme]{BCG1} that 
\begin{align}\label{old_part_1}
\det((\id-H)/(n-1))\geq n^{(1-n)}\det(H)^{1/n}\geq n^{-\frac{n(n+d-2)}{2(n-1)}}\det(H)^{\frac{n-d}{2(n-1)}}
\end{align}
 for any positive semi-definite matrix $H$ with $\tr H=1$. The second inequality in \eqref{old_part_1} follows from $\det(H)\leq n^{-n}$ as an application of the arithmetric mean-geometric mean inequality. By Property \hyperlink{P2}{(P2)} for the complex (or quaternionic) hyperbolic setting and again the log-concavity of determinant function for positive semi-definite matrices, we have
\begin{align}\label{old_part_2}\det((\widehat H-H)/(d-1))\geq \det(H).\end{align}
Therefore the complex (or quaternionic) hyperbolic version of \eqref{simplification} follows from \eqref{old_trick}, \eqref{old_part_1} and \eqref{old_part_2}.

\textbf{Going back to the Cayley hyperbolic setting.} The argument above does not work directly for the Cayley hyperbolic case due to the lack of Property \hyperlink{P2}{(P2)}. Namely, it is not obvious that $\det((\widehat{H}-H)/(d-1))=\det((\widehat{H}-H)/7)\geq\det H$ ($d=\dim_\RR\OO=8$) because we cannot write $\widehat{H}-H$ as the sum of $7$ conjugates of $H$. (See the remark after Corolllary \ref{curv_data_OH^2} or Subsection \ref{s3}.) However, we can make some adjustments to the above argument so that it can also work (with a lot more technical details) in the Cayley hyperbolic case. Set 
$$U=\matii{0&C}{C^*&0}\mathrm{~and~}U_\alpha=\alpha U.$$
for some $0\leq \alpha\leq1$ to be determined. By the matrix expressions \eqref{matrix_formulation} for $H$ and $\widehat H-H$, we can apply the arithmetric mean-geometric mean inequality and obtain
\begin{align}\label{?}\det\left[\frac{\widehat H-(H-U)}{7}\right]=&\det\left(\frac{1}{7}\matii{\lambda\cdot\id-A&0}{0&\mu\cdot\id-B}\right) \nonumber\\
\geq&\det\matii{A&0}{0& B}\geq\det\matii{A&C}{C& B}=\det(H).
\end{align}
In other words, the main reason \eqref{old_part_2} is not clear in the Cayley hyperbolic setting comes from the ``troublesome'' $U$. Moreover, since both $H$ and $\widehat H$ are positive semi-definite, one can verify that 
\begin{align}\label{??}\det(\widehat H-H+U_\alpha)=\det(\lambda\cdot\id-A)\det[\mu\cdot\id-B-(1-\alpha)^2C^*(\lambda\cdot\id-A)^{-1}C]\end{align} 
is an increasing function for $0\leq \alpha\leq 1$. This is because of the following linear algebra fact which follows from computations. We will use this fact extensively in the rest of this paper.
\begin{fact}\label{FACT}
Let $Q_1, Q_2$ be $n\times n$ symmetric matrices and $P$ an $n\times n$ matrix. Assume that $Q_1$ is invertible. Then
$$\matii{Q_1& P}{P^*& Q_2}=\matii{\id&0}{P^*Q_1^{-1}&\id}\matii{Q_1&0}{0&Q_2-P^*Q_1^{-1}P}\matii{\id&Q_1^{-1}P}{0&\id}.$$
Hence $\matii{Q_1& P}{P^*& Q_2}$ is positive (semi-)definite if and only if both $Q_1$ and $Q_2-P^*Q_1^{-1}P$ are positive (semi-) definite. Also we have the determinant equality
$$\det\matii{Q_1& P}{P^*& Q_2}=\det(Q_1)\det(Q_2-P^*Q_1^{-1}P).$$
\end{fact}
From \eqref{?} and \eqref{??} one can expect that for some $0\leq \alpha\leq 1$ which is likely close to $1$, we have
\begin{align}\label{new_part_2}
\det[(\widehat{H}-H+U_\alpha)/7]\geq \det(H).
\end{align}
Recall that it is not clear whether \eqref{old_part_2} holds in the Cayley hyperbolic setting. We will use \eqref{new_part_2} instead of \eqref{old_part_2} for our proof of \eqref{simplification}. 

We will also use the log-concavity of the determinant function for positive semi-definite matrices similar to \eqref{old_trick}. Namely, 
\begin{align}\label{new_trick}\det(\id-H+(\widehat{H}-H))\geq 22^{16}\det((\id-H-U_\alpha)/15)^{15/22}\det((\widehat{H}-H+U_\alpha)/7)^{7/22}.\end{align}
This inequality \eqref{new_trick} will play a similar role in the proof of Proposition \ref{ineq_prop} as \eqref{old_trick} does in \hyperlink{BCG's proof sketch}{Besson-Courtois-Gallot's proof} for \eqref{simplification} in the complex (or quaternionic) hyperbolic setting. (See \eqref{old_part_1}, \eqref{old_part_2} and \eqref{old_trick}.) The last remaining building-block which corresponds to \eqref{old_part_1} is to prove 
\begin{align}\label{new_part_1}\det(\id-H-U_\alpha)\geq\frac{15^{16}}{16^{16(1-4/15)}}\det(H)^{4/15}\end{align}
for the same $\alpha$ in \eqref{new_part_2}. Given \eqref{old_part_1} which is the same as \eqref{new_part_1} in the Cayley hyperbolic setting when $\alpha=0$, one might guess that this $\alpha$ should be close to $0$. We summarize all the reductions discussed above as the following idea of the proof for \eqref{simplification} and hence for Proposition \ref{ineq_prop}.
\end{proof}
\begin{proof}[Idea of the proof for Proposition \ref{ineq_prop}]
Recall that proving \eqref{simplification} is equivalent to proving Proposition \ref{ineq_prop}.  Notice that \eqref{new_trick} holds for any $\alpha\in [0,1]$. If there exists some $\alpha\in[0,1]$ such that \eqref{new_part_1} and \eqref{new_part_2} hold, then \eqref{simplification} is a direct corollary of \eqref{new_trick}, \eqref{new_part_1} and \eqref{new_part_2}. Therefore it remains for us to solve the following linear algebra problem: \textbf{Finding an $\alpha\in [0,1]$ such that both \eqref{new_part_1} and \eqref{new_part_2} hold.} 
\end{proof}
Finding such an $\alpha$ is also the main difficulty in this proof since naively we expect that \eqref{new_part_1} holds when $\alpha$ is ``close'' to $0$ and \eqref{new_part_2} holds when $\alpha$ is ``close'' to $1$. Therefore we need to verify the existence of an $\alpha$ which is ``close'' to both $0$ and $1$ in the above sense. 

Lemma \ref{key_lemma} in Subsection \ref{s5} confirms the existence of such an $\alpha$ and hence completes the proof of Proposition \ref{ineq_prop}.
\subsection{Repairing the gap: a linear algebra lemma}\label{s5}
Recall that from the end of Subsection \ref{s4}, it remains for us to find an $\alpha\in[0,1]$ such that \eqref{new_part_1} and \eqref{new_part_2} hold. This section gives a positive answer to the existence of such an $\alpha$. This can be summarized as the following linear algebra lemma.
\begin{lem}\label{key_lemma}
Following the notations from Subsection \ref{s4}, there exists some $0\leq\alpha\leq 1$ such that
\begin{enumerate}
\item[(1).] $\det(\widehat{H}-H+U_\alpha)\geq 7^{16}\det(H)$ (see also \eqref{new_part_2});
\item[(2).] $\det(\id-H-U_\alpha)\geq\frac{15^{16}}{16^{16(1-4/15)}}\det(H)^{4/15}$ (see also \eqref{new_part_1}).
\end{enumerate}
\end{lem}
\begin{proof}
Let $\lambda_1\geq...\geq\lambda_8\geq0$ and $\mu_1\geq...\geq\mu_8$ be eigenvalues (and therefore the diagonal elements) of $A$ and $B$ respectively. 

Without loss of generality we assume that 
$$A=\matiii{\lambda_1&&}{&\ddots&}{&&\lambda_8}\mathrm{~and~}B=\matiii{\mu_1&&}{&\ddots&}{&&\mu_8}.$$
Write $C=(c_{ij})_{1\leq i,j\leq8}$. Denote by $C_k=(c_{ki}c_{kj})_{1\leq i,j\leq 8}$ with $1\leq k\leq 8$. Then for any diagonal matrix $D=\mathrm{diag}(d_1,...,d_8)$, we have 
\begin{align}\label{decomp_C}C^*DC=\left(\sum_{k=1}^8c_{ki}d_kc_{kj}\right)_{1\leq i,j\leq 8}=\sum_{k=1}^8 d_kC_k.\end{align}
When $C_k\neq 0$, we can think of $C_k/\tr C_k$ as the matrix for the orthogonal projection onto the $k$-th row vector of $C$, where $(\tr C_k)^{1/2}$ equal to the norm of this vector. It is clear that $C_k$ has rank at most $1$. Then \eqref{decomp_C} states the linear algebra fact that any matrix of the form $C^*DC$ with $D$ diagonal is a linear combination of orthogonal projections onto the row vectors of $C$. Applying Fact \ref{FACT} and \eqref{decomp_C} to $H$ and $\widehat{H}-H$ yields
\begin{align}\label{**}B-C^*A^{-1}C=B-\left(\sum_{k=1}^8C_k/\lambda_k\right)\geq 0\end{align}
and 
\begin{align}\label{***}\mu\cdot\id-B-C^*(\lambda-A)^{-1}C=\mu\cdot\id-B-\left(\sum_{k=1}^8C_k/(\lambda-\lambda_k)\right)\geq 0.\end{align}

The rest of the proof is divided into two cases. We will first discuss the ``good case''.

\textbf{Motivations for the ``good case'': Find conditions on $A,B$ when $\alpha=0$ works.} 

Recall that in \hyperlink{BCG's proof sketch}{Besson-Courtois-Gallot's proof} for \eqref{simplification} in the complex (or quaternionic) hyperbolic setting, (see \eqref{old_part_1}, \eqref{old_part_2} and \eqref{old_trick}), we can choose $\alpha=0$ and the corresponding version of Lemma \ref{key_lemma} holds. We want to first discuss to what extent $\alpha=0$ works in the Cayley hyperbolic setting. Notice that \eqref{old_part_1} (see \cite[B.4 Lemme]{BCG1}) is true for any positive semi-definite matrix $H$ with trace 1, the second inequality in Lemma \ref{key_lemma} always holds when $\alpha=0$. Therefore it suffices to find the conditions when the first inequality holds and $\alpha=0$ simultaneously. By Fact \ref{FACT} we have
\begin{align}\label{H_det_decomp}
\det(H)=\det(A)\det(B-C^*A^{-1}C)
\end{align}
and 
\begin{align}\label{hatH-H_det_decomp}
\det(\widehat H-H)=\det(\lambda\cdot\id-A)\det(\mu\cdot\id-B-C^*(\lambda\cdot\id-A)^{-1}C).
\end{align}
Since $A$ and $B$ are both positive semi-definite and have trace $\lambda$ and $\mu$ respectively, one can apply the arithmetic mean-geometric mean inequality and obtain
\begin{align}\label{am_gm_good}\det(\lambda\cdot\id-A)\geq 7^8\det(A)\mathrm{~and~}\det(\mu\cdot\id-B)\geq 7^8\det(B).\end{align}
Therefore if we can prove that
$$\frac{\det(\widehat H-H)}{\det(\lambda\cdot\id-A)\det(\mu\cdot\id-B)}\geq\frac{\det(H)}{\det(A)\det(B)},$$
then the first inequality of Lemma \ref{key_lemma} holds when $\alpha=0$. This is equivalent to proving
\begin{align}\label{<<<}  
\frac{\det(\mu\cdot\id-B-C^*(\lambda\cdot\id-A)^{-1}C)}{\det(\mu\cdot\id-B)}\geq\frac{\det(B-C^*A^{-1}C)}{\det(B)}.
\end{align}
One can naively expect \eqref{<<<} to be true if $\mu\cdot\id-B$ is more positive than $B$ and $C^*(\lambda\cdot\id-A)^{-1}C$ is less positive than $C^*A^{-1}C$. This comes from the observation that if all the matrices in \eqref{<<<} are of size $1\times 1$, then it reduces to the elementary fact that $(a-d)/a\geq (b-c)/c$ for any positive numbers $a\geq b\geq c\geq d>0$. Therefore we discuss the following ``good'' case.

\textbf{Case 1:} If $2\lambda_1\leq\lambda$ and $2\mu_1\leq \mu$, then  we have $\mu\cdot\id - B\geq B$ and 
$$C^*A^{-1}C=\left(\sum_{k=1}^8C_k/\lambda_k\right)\geq\left(\sum_{k=1}^8C_k/(\lambda-\lambda_k)\right)=C^*(\lambda\cdot\id-A)^{-1}C,$$
where two symmetric matrices $Q_1\geq Q_2$ if and only if $Q_1-Q_2$ is positive semi-definite. Therefore
\begin{align}\label{*}0\leq \mu\cdot\id-B-\left(\sum_{k=1}^8C_k/\lambda_k\right)\leq\mu\cdot\id-B-\left(\sum_{k=1}^8C_k/(\lambda-\lambda_k)\right).\end{align}
Hence
\begin{align}\label{half_done_ineq_1}
\frac{\det(B-C^*A^{-1}C)}{\det(\mu\cdot\id-B-C^*(\lambda\cdot\id-A)^{-1}C)} 
\leq\frac{\det(B-C^*A^{-1}C)}{\det(\mu\cdot\id-B-C^*A^{-1}C)}
\end{align}
Before we finish the estimate, we introduce a linear algebra lemma.
\begin{lem}\label{easy_lin_alg}
Let $H_1\geq H_2\geq W\geq 0$ be $n\times n$ positive semi-definite matrices. Then
$$\det(H_2-W)\det(H_1)\leq\det(H_1-W)\det(H_2).$$
In particular, if $n=1$, then we have the naive inequality $(H_2-W)H_1\leq(H_1-W)H_2$ for any non-negative real numbers $H_1\geq H_2\geq W\geq 0$.
\end{lem}
\begin{proof}[Proof of Lemma \ref{easy_lin_alg}]
If one of $H_1$, $H_2$ is singular then the result is trivial. Without loss of generality we assume that both $H_1,H_2$ are invertible and $H_1=\id$. Since $W$ can be written as the sum of finitely many rank $1$ matrices, we can also assume WLOG that $W=\mathrm{diag}\{c,0,...,0\}$. Write $H_2=(h_{ij})_{1\leq i,j,\leq n}$. Define 
$$\widehat{H}_2=\mativ{h_{11}&0&...&0}{0&h_{22}&...&h_{2n}}{\vdots&\vdots&\ddots&\vdots}{0&h_{n2}&...&h_{nn}}.$$
Then $\det{H_2}\leq\det{\widehat{H}_2}$ and 
$$\det(\widehat{H}_2)-\det(\widehat{H}_2-W)=\det(H_2)-\det(H_2-W)$$
Hence
\begin{align*}\frac{\det(H_2-W)}{\det(H_2)}\leq\frac{\det(\widehat{H}_2-W)}{\det(\widehat{H}_2)}=\frac{h_{11}-c}{h_{11}}\leq\frac{1-c}{1}=\frac{\det(H_1-W)}{\det(H_1)}.\tag*{\qedhere}\end{align*}
\end{proof}
Choose $H_1=\mu\cdot\id-B$, $H_2=B$ and $W=C^*A^{-1}C$ in Lemma \ref{easy_lin_alg}, following \eqref{am_gm_good} and \eqref{half_done_ineq_1} we have
\begin{align*}
\frac{\det(B-C^*A^{-1}C)}{\det(\mu\cdot\id-B-C^*(\lambda-A)^{-1}C)} \leq\frac{\det(B)}{\det(\mu\cdot\id-B)}\leq 7^{-8}.
\end{align*}
Hence by \eqref{am_gm_good},
\begin{align}\label{good_case}
\frac{\det(H)}{\det(\widehat{H}-H)}\leq\frac{\det(A)\det(B)}{\det(\lambda\cdot\id-A)\det(\mu\cdot\id-B)}\leq 7^{-16}.
\end{align}
Choose $\alpha=0$ and the proof of Lemma \ref{key_lemma} in this case follows from \cite[B.3-B.4 Lemmes]{BCG1}.

It remains for us to prove Lemma \ref{key_lemma} in the following ``bad'' case.

\textbf{\hypertarget{case2}{Case 2}:} If $2\lambda_1>\lambda$ or $2\mu_1>\mu$, without loss of generality we assume that $2\mu_1>\mu$. This case is trickier since $\mu\cdot\id-B$ is no longer more positive than $B$ and $C^*(\lambda\cdot\id-A)^{-1}C$ is no longer less positive than $C^*A^{-1}C$. Therefore the arguments in the previous case which are slight adjustments to \hyperlink{BCG's proof sketch}{Besson-Courtois-Gallot's approach} (see also \cite[B.3-B.4 Lemmes]{BCG1}) do not work in this case. Hence choosing $\alpha=0$ may not work for \eqref{new_trick} and Lemma \ref{key_lemma}. ($\alpha=0$ could still work in this case but it is not obvious how to verify it.)  Another reason why this case is ``bad'' (or actually ``good'' depending on one's perspective) comes from the following observation: Equality for the first inequality in Lemma \ref{key_lemma} holds when $\lambda_1=...=\lambda_8$, $\mu_1=...=\mu_8$ and $C=0$ for every $\alpha$. The assumptions on eigenvalues $\lambda_j$'s and $\mu_j$'s for $A$ and $B$ in this case are very far from the above mentioned equality conditions. Therefore one can expect that too much difference between eigenvalues of $A$ and $B$ will be the main reason why the inequalities in Lemma \ref{key_lemma} hold in this case. In addition to that, one can expect that equalities will not be achieved here.

Since we need to find an $\alpha$ satisfying $2$ inequalities in Lemma \ref{key_lemma}, we split the rest of the proof into $2$ steps.

\textit{Step 1: Find the conditions for $\alpha$ such that the first inequality in Lemma \ref{key_lemma} holds in} \textbf{\hyperlink{case2}{Case 2}}.

By Fact \ref{FACT}, we have
$$\det(\widehat H-H+U_\alpha)=\det(\lambda\cdot\id-A)\det(\mu\cdot\id-B-(1-\alpha)^2C^*(\lambda\cdot\id-A)^{-1}C).$$
Recall that by \eqref{am_gm_good} we have $\det(\lambda\cdot\id-A)\geq7^8\det(A).$ In order to satisfy the first requirement, it suffices to find $\alpha$ such that
\begin{align}
\label{ineq1}\det(\mu\cdot\id-B-(1-\alpha)^2C^*(\lambda\cdot\id-A)^{-1}C)\geq 7^8\det(B-C^*A^{-1}C).
\end{align}
 
\textbf{Main idea to find a lower bound for the left hand side of \eqref{ineq1}:} Usually there is no obvious way to give a non-trivial lower bound on the determinant of a positive semi-definite matrix (e.g. the left hand side of \eqref{ineq1}). However we notice that the matrix $\mu\cdot\id-B-(1-\alpha)^2C^*(\lambda\cdot\id-A)^{-1}C$ has lower bounds on its trace and its smallest eigenvalue. If we view its determinant as a function on its eigenvalues, then this function will have small outputs when all the eigenvalues deviate a lot from their average. This deviation is controlled because of the bounds on the trace and the smallest eigenvalues. Therefore the determinant cannot be arbitarily small. A detailed argument realising the above observations is presented below.
 
\textbf{Details for finding a lower bound for the left hand side of \eqref{ineq1}:} Since $\lambda-\lambda_k\geq\lambda_k$ for any $2\leq k\leq 8$, \eqref{**} implies that 
\begin{align}\label{case2_trace_ineq}
\tr\left(\sum_{k=2}^8C_k/(\lambda-\lambda_k)\right)\leq \tr\left(\sum_{k=2}^8C_k/\lambda_k\right)\leq \tr(C^*A^{-1}C)\leq \tr B=\mu.
\end{align}
Hence we set 
\begin{align}\label{beta_def}
\tr\left(\sum_{k=2}^8C_k/(\lambda-\lambda_k)\right)=\beta\mu
\end{align}
for some $0\leq\beta\leq 1$. On the other hand, since $\mu\cdot\id-B$ has norm bounded above by $\mu$ and $C_1$ has rank $1$, by \eqref{***} we have
$$\tr(C_1/(\lambda-\lambda_1))=\|C_1/(\lambda-\lambda_1)\|\leq\|C^*(\lambda\cdot\id-A)^{-1}C\|\leq\|\mu\cdot\id-B\|\leq \mu.$$
Therefore the matrix $\mu\cdot\id-B-(1-\alpha)^2C^*(\lambda\cdot\id-A)^{-1}C$ (introduced in \eqref{***}) has trace at least $[7-(1+\beta)(1-\alpha)^2]\mu$.
Notice that \eqref{***} implies
$$\mu\cdot\id-B-(1-\alpha)^2C^*(\lambda\cdot\id-A)^{-1}C\geq\mu\cdot\id-B-(1-\alpha)^2(\mu\cdot\id-B)=[1-(1-\alpha)^2](\mu\cdot\id-B).$$
Hence the smallest eigenvalue of $\mu\cdot\id-B-(1-\alpha)^2C^*(\lambda\cdot\id-A)^{-1}C$ is greater or equal to $[1-(1-\alpha)^2](\mu-\mu_1)$, where $\mu-\mu_1$ is the smallest eigenvalue of $\mu\cdot\id-B$.  We summarize these estimates on trace and the smallest eigenvalue for reader's convenience.
\begin{enumerate}
\item[\hypertarget{E1}{(E1)}.] Lower bound for trace: $\tr[\mu\cdot\id-B-(1-\alpha)^2C^*(\lambda\cdot\id-A)^{-1}C]\geq[7-(1+\beta)(1-\alpha)^2]\mu$, where $\beta\in[0,1]$ is given in \eqref{beta_def};
\item[\hypertarget{E2}{(E2)}.] Lower bound for the smallest eigenvalue: The smallest eigenvalue of $\tr[\mu\cdot\id-B-(1-\alpha)^2C^*(\lambda\cdot\id-A)^{-1}C]$ is greater or equal to $[1-(1-\alpha)^2](\mu-\mu_1)$.
\end{enumerate}

Since $\mu\cdot\id\geq\mu\cdot\id-B-(1-\alpha)^2C^*(\lambda\cdot\id-A)^{-1}C$, the norm of $\id-B-(1-\alpha)^2C^*(\lambda\cdot\id-A)^{-1}C$ is bounded above by $\mu$. Notice that
\begin{align*}
&5\mu+3[1-(1-\alpha)^2](\mu-\mu_1)\\
<& 5\mu+1.5[1-(1-\alpha)^2]\mu =[6.5-1.5(1-\alpha)^2]\mu\leq [7-2(1-\alpha)^2]\mu\leq[7-(1+\beta)(1-\alpha)^2]\mu.
\end{align*}
By \hyperlink{E1}{(E1)} and \hyperlink{E2}{(E2)}, this suggests that the sum of $3$ copies of the smallest eigenvalue and $5$ copies of the largest eigenvalue (which is bounded above by $\mu$) of $\mu\cdot\id-B-(1-\alpha)^2C^*(\lambda\cdot\id-A)^{-1}C$ is strictly smaller than its trace. Therefore $\mu\cdot\id-B-(1-\alpha)^2C^*(\lambda\cdot\id-A)^{-1}C$ admits at most $2$ eigenvalues equal to its smallest eigenvalue. The following  linear algebra lemma explains how to obtain a lower bound on its determinant given the above observation.
\begin{lem}\label{reverse_AM_GM_ineq_1}
Let $Q$ be an $8\times 8$ symmetric matrix with real entries such that its eigenvalues $q_1,...,q_8$ lie in $[m,M]$ for some $0<m\leq M$. Assume that there exists some constant $K$ such that $\tr(Q)\geq K>3m+5M$, then
$$\det(Q)\geq m^2(K-2m-5M)M^5.$$
\end{lem}
\begin{proof}[Proof of Lemma \ref{reverse_AM_GM_ineq_1}]
WLOG we assume $m\leq q_1\leq...\leq q_8\leq M$. We define a function $P(x_1,...,x_8)=x_1x_2...x_8$ with domain defined by $x_1+...+x_8=\tr(Q)$ and $m\leq x_1\leq...\leq x_8\leq M$. Assume that $P$ achieves global minimum at $(y_1,...,y_8)$ in its domain. Then there exists at most one $y_i\not\in\{m,M\}$. (This is because if there exists $1\leq i<j\leq 8$ such that $y_i,y_j\not\in\{m,M\}$, WLOG we can assume that $y_l=m$ for $l<i$ and $y_k=M$ for $k>j$. Then we can choose $0<\epsilon\ll 1$ such that $m<y_i-\epsilon<y_j+\epsilon<M$. It is not hard to verify that $(y_1,...,y_i-\epsilon,...,y_j+\epsilon,...,y_8)$ is still in the domain of $P(x_1,...,x_8)$ and $P(y_1,...,y_i-\epsilon,...,y_j+\epsilon,...,y_8)<P(y_1,...,y_8)$, which contradicts with the assumption on $(y_1,...,y_8)$. ) Since $\tr(Q)\geq K>3m+5M$, we have $y_4=...=y_8=M$.

\textbf{Case (1):} If $y_1=y_2=m$, then 
$$\det(Q)=P(q_1,...,q_8)\geq P(y_1,...,y_8)=m^2(\tr(Q)-2m-5M)M^5\geq m^2(K-2m-5M)M^5.$$

\textbf{Case (2):} If $y_2>m$ and $y_1=m$, then $m< \tr(Q)-m-6M\leq M$ and
$$\det(Q)\geq P(y_1,...,y_8)=m(\tr(Q)-m-6M)M^6>m^2(\tr(Q)-2m-5M)M^5\geq m^2(K-2m-5M)M^5.$$

\textbf{Case (3):} If $y_1>m$, then $m< \tr(Q)-7M\leq M$ and 
$$\det(Q)\geq P(y_1,...,y_8)=(\tr(Q)-7M)M^7>m^2(\tr(Q)-2m-5M)M^5\geq m^2(K-2m-5M)M^5.$$
Here the strict inequality follows from the fact that $\tr(Q)-m-6M\geq M$ and that
\begin{align*}(\tr(Q)-7M)M^2> m(\tr(Q)-m-6M)M>m^2(\tr(Q)-2m-5M). \tag*{\qedhere}\end{align*}
\end{proof}
Applying Lemma \ref{reverse_AM_GM_ineq_1} to the case when $Q=\mu\cdot\id-B-(1-\alpha)^2C^*(\lambda\cdot\id-A)^{-1}C$, $m=[1-(1-\alpha)^2](\mu-\mu_1)$, $M=\mu$ and $K=[7-(1+\beta)(1-\alpha)^2]\mu$, we have
\begin{align}\label{deno_est}
&\det(\mu\cdot\id-B-(1-\alpha)^2C^*(\lambda\cdot\id-A)^{-1}C) \nonumber\\
\geq&\{[1-(1-\alpha)^2](\mu-\mu_1)\}^2\cdot \{2\mu_1[1-(1-\alpha)^2]+(1-\beta)(1-\alpha)^2\mu\}\cdot \mu^5 \nonumber\\
=&[1-(1-\alpha)^2]^2(\mu-\mu_1)^2\{2\mu_1[1-(1-\alpha)^2]+(1-\alpha)^2\mu-\beta(1-\alpha)^2\mu\}\mu^5. 
\end{align}
It remains for us to prove that the lower bound in \eqref{deno_est} is greater equal to $\det(B-C^*A^{-1}C)$. We will do so by finding a suitabale upper bound for $\det(B-C^*A^{-1}C)$ (equivalently for the right hand side of \eqref{ineq1}). Similar to the arguments on finding a lower bound for the left hand side of \eqref{ineq1}, we first summarize the idea behind finding this upper bound estimate which is less complicated than the actual arguments. 

\textbf{Main idea to find an upper bound for the right hand side of \eqref{ineq1}:} Equivalently we want to find an upper bound for $\det(B-C^*A^{-1}C)$. Since we have assumed $\mu_1\geq\mu/2$ in this case, the possibilities of $B-C^*A^{-1}C$ having equal eigenvalues are very low when $C$ is close to $0$. Therefore directly applying the arithmetic mean-geometric mean inequality to the eigenvalues of $B-C^*A^{-1}C$ may overkill. We want to use the idea of the arithmetic mean-geometric mean inequality in a more delicate way. Namely, if trace of $B-C^*A^{-1}C$ is fixed, $\det(B-C^*A^{-1}C)$ becomes larger when its eigenvalues are closer to each other, even in the case when it is impossible for all eigenvalues to be the same. In particular, we noticed that if $C$ is close to $0$, the largest possible value of $\det(B-C^*A^{-1}C)$ as a function on its eigenvalues is achieved when $B$ and $B-C^*A^{-1}C$ have the same eigenvalues except the largest one. The main reason is that $\mu_1$ is way larger than other eigenvalues $\mu_j$ of $B$ in this case. The technical details of the above idea is presented below.

\textbf{Details for finding an upper bound for the right hand side of \eqref{ineq1}:} By \eqref{**}, \eqref{case2_trace_ineq} and \eqref{beta_def} we have $B-C^*A^{-1}C$ is positive semi-definite with trace smaller or equal to $(1-\beta)\mu$. Since $B\geq B-C^*A^{-1}C$, we can let $\mu_j-b_j$ be the $j$-th diagonal entry of $B-C^*A^{-1}C$. By Fact \ref{FACT} we know that the determinant of any positive semi-definite matrix is smaller or equal to the product of its diagonal elements. In particular, $\det(B-C^*A^{-1}C)\leq\prod_{j=1}^8(\mu_j-b_j)$.

Define 
$$E_\beta(z_1,...,z_8)=\prod_{j=1}^8\left(1-\frac{z_j}{\mu_j}\right),\quad, 0\leq z_j\leq \mu_j, \sum_{j=1}^8z_j\geq \beta\mu.$$
Then we have $0\leq \mu_j-b_j\leq \mu_j$ and $b_1+...+b_8\geq \beta\mu$ by \eqref{**}, \eqref{case2_trace_ineq} and \eqref{beta_def}. Hence $(b_1,...,b_8)$ is in the domain of $E_\beta$ and
$$\det(B-C^*A^{-1}C)\leq\det(B)E_\beta(b_1,...,b_8)\leq\det(B)\cdot\sup E_\beta\leq \mu_1\left(\frac{\mu-\mu_1}{7}\right)^7\sup E_\beta.$$
The following lemma provides a more delicate arithmetic mean-geometric mean estimate regarding $E_\beta$.  
\begin{lem}\label{fact2} Under the above assumptions, $E_\beta$ achieves maximum at a point $(\zeta_1,...,\zeta_8)$ satisfying the following additional conditions.
\begin{enumerate}
\item[(1).] $\zeta_1+...+\zeta_8=\beta\mu$;
\item[(2).] $\mu_1-\zeta_1\geq...\geq\mu_8-\zeta_8$;
\item[(3).] If $\zeta_j\neq 0$, then $\mu_1-\zeta_1=\mu_j-\zeta_j$.
\end{enumerate}
\end{lem}
\begin{rmk}
One way to understand these conditions is to vary $\beta$ from $0$ to $1$. When $\beta$ is small, $\zeta_j=0$ for $j\geq 2$ and hence $\zeta_1=\beta\mu$. $\zeta_2$ will not become positive until $\beta$ become too large that $\mu_1-\beta\mu\leq\mu_2$. Naively, if we think of $\mu_j$ as the ``pre-tax income of the $j$-th person'' and $\zeta_j$ as ``the tax he/she has to pay'', the process of finding $\sup E_\beta$ shows an example of ``the one who earns more pays more taxes'' with ``the least possible inequality in after-tax income''. This explains its similarity to the arithmetic mean-geometric mean inequality since we are maximizing $E_\beta$ by minimizing the differences among $\mu_j-\zeta_j$ under the given conditions. 
\end{rmk}
\begin{proof}[Proof of Lemma \ref{fact2}]
Since the domain of $E_\beta$ is compact, the global maximum can be achieved at some point $(\zeta_1,...,\zeta_8)$. We verify all three conditions by contradictions.
\begin{enumerate}
\item[(1).] If $\zeta_1+...+\zeta_8=:(1+\epsilon)\beta\mu>\beta\mu$ for some $\epsilon>0$, then $E_\beta(\zeta_1/(1+\epsilon),...,\zeta_8/(1+\epsilon))>E_\beta(\zeta_1,...,\zeta_8)$. This contradicts our assumption on $(\zeta_1,...,\zeta_8)$. Therefore the first condition must hold;
\item[(2).] If $\mu_i-\zeta_i<\mu_j-\zeta_j$ for some $i<j$, we define
$$\zeta'_k=\begin{cases}\displaystyle \zeta_k,~~&k\neq i,j;\\ \mu_k-\frac{\mu_i-\zeta_i+\mu_j-\zeta_j}{2} &k=i\mathrm{~or~}j.\displaystyle \end{cases}$$
Then we have $\mu_i>\mu_j\geq\mu_j-\zeta_j>\mu_i-\zeta'_i=\mu_j-\zeta'_j=(\mu_i-\zeta_i+\mu_j-\zeta_j)/2>0$. Hence $(\zeta'_1,...\zeta'_8)$ is in the domain of $E_\beta$ and $E_\beta(\zeta'_1,...\zeta'_8)>E_\beta(\zeta_1,...\zeta_8)$ by the arithmetic mean-geometric mean inequality applied to $(\mu_i-\zeta_i)(\mu_j-\zeta_j)$. This contradicts our assumption on $(\zeta_1,...,\zeta_8)$. Therefore the second condition must hold;
\item[(3).] If $\zeta_j> 0$ and $\mu_1-\zeta_1>\mu_j-\zeta_j$ for some $j\geq 2$, WLOG we can assume that $\mu_1-\zeta_1=\mu_k-\zeta_k$ for any $k<j$. Then we can choose some $0<\epsilon<\zeta_j$ such that $\mu_{j-1}-(\zeta_{j-1}+\epsilon)\geq\mu_j-(\zeta_j-\epsilon)$. Notice that
$[\mu_{j-1}-(\zeta_{j-1}+\epsilon)][\mu_j-(\zeta_j-\epsilon)]>(\mu_{j-1}-\zeta_{j-1})(\mu_j-\zeta_j)$, we have
$$E_\beta(\zeta_1,...,\zeta_{j-1}+\epsilon,\zeta_{j}-\epsilon,...,\zeta_8)>E_\beta(\zeta_1,...,\zeta_8).$$ This contradicts our assumption on $(\zeta_1,...,\zeta_8)$. Therefore the third condition must hold. \qedhere
\end{enumerate}
\end{proof}

Now we will give an upper bound for $\sup E_\beta$ given the above lemma. Notice that $\mu_1-\zeta_1\leq\sum_{j=2}^8\mu_j-\zeta_j$ if $(\zeta_2,...,\zeta_8)\neq 0$ and that $\mu_1\geq\sum_{j=2}^8\mu_j$. Therefore $\zeta_1\geq\beta\mu/2\geq\sum_{j=2}^8\zeta_j$ and
$$\sup E_\beta=\prod_{j=1}^8\frac{\mu_j-\zeta_j}{\mu_j}\leq\frac{\mu_1-\zeta_1}{\mu_1}\leq\frac{\mu_1-\frac{\beta\mu}{2}}{\mu_1}.$$
Hence
\begin{align}\label{num_est}
\det(B-C^*A^{-1}C)\leq \mu_1\left(\frac{\mu-\mu_1}{7}\right)^7\sup E_\beta\leq\left(\mu_1-\frac{\beta\mu}{2}\right)\left(\frac{\mu-\mu_1}{7}\right)^7.
\end{align}
We are now in a position to find conditions on $\alpha$ such that \eqref{ineq1} holds. Given \eqref{deno_est} and \eqref{num_est}, it suffices to find $\alpha$ such that
\begin{align}\label{upp_bd_est_*}
&[1-(1-\alpha)^2]^2(\mu-\mu_1)^2\{2\mu_1[1-(1-\alpha)^2]+(1-\alpha)^2\mu-\beta(1-\alpha)^2\mu\}\mu^5\nonumber \\
\geq&7\left(\mu_1-\frac{\beta\mu}{2}\right)(\mu-\mu_1)^7.
\end{align}
We want to apply a few additional requirements on $\alpha$ to slim down the expression. For simplicity we assume in addition that $\alpha\geq1-\sqrt{2/3}$. Then $(1-\alpha)^2\leq 2/3$ and $[1-(1-\alpha)^2]\geq 1/3$. Therefore
\begin{align*}
2\mu_1[1-(1-\alpha)^2]+(1-\alpha)^2\mu=&2\mu_1-(1-\alpha)^2(2\mu_1-\mu) \\
\geq&2\mu_1-(1-\alpha)^2\mu_1\geq2\mu_1-2\mu_1/3=4\mu_1/3\geq 2(1-\alpha)^2\mu_1.
\end{align*}
This implies that
\begin{align}\label{alpha1}
\frac{2\mu_1[1-(1-\alpha)^2]+(1-\alpha)^2\mu-\beta(1-\alpha)^2\mu}{\mu_1-\frac{\beta\mu}{2}}\geq 2(1-\alpha)^2,~~\forall~0\leq\beta\leq 1.
\end{align}
The idea behind assuming $\alpha\geq1-\sqrt{2/3}$ comes from some ``backward thinking''. We want to make \eqref{alpha1} hold in order to have a nicer expression to work with. Then this leads to the assumption that $\alpha\geq1-\sqrt{2/3}$. Let $\tau=(\mu-\mu_1)/\mu\in[0,1/2)$, then by \eqref{alpha1} applied to \eqref{upp_bd_est_*} it suffices to find an $\alpha\geq1-\sqrt{2/3}$ such that
$$[1-(1-\alpha)^2]^2\geq \frac{7\tau^5}{2(1-\alpha)^2}\geq\frac{\mu_1-\frac{\beta\mu}{2}}{2\mu_1[1-(1-\alpha)^2]+(1-\alpha)^2\mu-\beta(1-\alpha)^2\mu}\cdot 7\tau^5.$$
Equivalently we want to find an $\alpha\geq1-\sqrt{2/3}$ such that
$$[1-(1-\alpha)^2]^2(1-\alpha)^2=[1-(1-\alpha)^2]^2-[1-(1-\alpha)^2]^3\geq\frac{7}{2^6}.$$
Notice that the function $x(t):=t^2-t^3$ is increasing on $[0,2/3]$ and decreasing on $[2/3,1]$. Moreover, $x(1/3)=8/27>7/2^6$ and $x(3/4)=9/2^6>7/2^6$. Therefore we can choose $\alpha$ such that
$$\frac{1}{3}\leq [1-(1-\alpha)^2]\leq\frac{3}{4}.$$
This is equivalent to
\begin{align}
\label{alpha_lower_bd}\frac{1}{2}\geq\alpha\geq 1-\sqrt{\frac{2}{3}}.
\end{align}
Hence any $\alpha$ satisfying \eqref{alpha_lower_bd} also satisfies the first requirement of Lemma \ref{key_lemma} in \textbf{\hyperlink{case2}{Case 2}}.

It remains for us to find conditions for $\alpha$ satisfying the second requirement of Lemma \ref{key_lemma}.

\textit{Step 2: Find the conditions for $\alpha$ such that the second inequality in Lemma \ref{key_lemma} holds in} \textbf{\hyperlink{case2}{Case 2}}. 

Recall that in \cite[B.4 Lemme]{BCG1} Besson-Courtois-Gallot proved that for any $1/n\leq s\leq 1$ the function 
$$H\to\frac{\det(H)^s}{\det(\id-H)},\quad \tr(H)=1\mathrm{~and~} H\geq 0$$
defined on $n\times n$ trace 1 positive semi-definite matrices achieves global maximum when $H=\id/n$. This proves \eqref{old_part_1} as a corollary. In particular, the second inequality in Lemma \ref{key_lemma} (see also  \eqref{new_part_1}) holds when $\alpha=0$. However, $\alpha=0$ does not satisfy \eqref{alpha_lower_bd}. For non-zero $\alpha$, the second inequality in Lemma \ref{key_lemma} may still hold in \textbf{\hyperlink{case2}{Case 2}} if $\alpha$ satisfies the following properties. 
\begin{enumerate}
\item[(1).] $\det(\id-H-U_\alpha)$ is ``not too small'' compared with $\det(\id-H)$. The lower bound between their ratio is dependent on $\alpha$;
\item[(2).] $\det(H)^{4/15}/\det(\id-H)$ is ``much smaller'' than $16^{16(1-4/15)}\cdot 15^{-16}=\det(\id/16)^{4/15}/\det(15\id/16)$ since $H$ differs a lot from $\id/16$ by the assumption of \textbf{\hyperlink{case2}{Case 2}}.  In particular, the upper bound between their ratio should be smaller than the ratio in the first property.
\end{enumerate}
To be more specific, we first rewrite the second inequality in Lemma \ref{key_lemma} as
\begin{align}\label{rewrite_new_part_1}\frac{\det(H)^{4/15}}{\det(\id-H-U_\alpha)}=\frac{\det(H)^{4/15}}{\det(\id-H)}\frac{\det(\id-H)}{\det(\id-H-U_\alpha)}\leq\frac{16^{16(1-4/15)}}{15^{16}}.
\end{align}
Based on \eqref{rewrite_new_part_1}, the main idea to find conditions on $\alpha$ is to find upper bounds $K_1(\alpha), K_2$ for $\det(\id-H)/\det(\id-H-U_\alpha)$ and $\det(H)^{4/15}/\det(\id-H)$ respectively. Then the conditions on $\alpha$ will be given in the form of \begin{align}\label{goal_of_step_3}K_1(\alpha)K_2\leq \frac{16^{16(1-4/15)}}{ 15^{16}}=\frac{\det(\id/16)^{4/15}}{\det(15\id/16)}.\end{align}

\textbf{Finding an upper bound $K_1(\alpha)$ for $\det(\id-H)/\det(\id-H-U_\alpha)$:} Since both $H$ and $\widehat H-H$ are positive semi-definite, by their matrix forms in \eqref{matrix_formulation} we have
$$\matii{A&C}{C^*&B},\matii{\lambda\cdot\id-A&-C}{-C^*&\mu\cdot\id-B},\matii{A&-C}{-C^*&B},\matii{\lambda\cdot\id-A&C}{C^*&\mu\cdot\id-B}$$
are positive semi-definite.
Therefore
$$\matii{\lambda\cdot\id& 2C}{2C^*&\mu\cdot\id}\geq \matii{A&C}{C^*&B}\geq 0,$$
Notice from Fact \ref{FACT} that 
$$\matii{\lambda\cdot\id& 2C}{2C^*&\mu\cdot\id}=\matii{\id&0}{2C^*/\lambda&\id}\matii{\lambda\cdot\id&0}{0&\mu\cdot\id-4\lambda^{-1}C^*C}\matii{\id&2C/\lambda}{0&\id},$$
We have $\lambda\mu\cdot\id\geq 4C^*C$, which implies that
$$\matii{\mu\cdot\id& -2C}{-2C^*&\lambda\cdot\id}=\matii{\id&0}{-2C^*/\mu&\id}\matii{\mu\cdot\id&0}{0&\lambda\cdot\id-4\mu^{-1}C^*C}\matii{\id&-2C/\mu}{0&\id}$$
is positive semi-definite. This implies that
\begin{align}\label{pos_def_mat_2}
\id-H-2U=\matii{\id-A&-3C}{-3C^*&\id-B}=\matii{\mu\cdot\id& -2C}{-2C^*&\lambda\cdot\id}+\matii{\lambda\cdot\id-A&-C}{-C^*&\mu\cdot\id-B}
\end{align}
is positive semi-definite. In particular,
$$\id-B\geq9C^*(\id-A)^{-1}C.$$
By Fact \ref{FACT} and \eqref{pos_def_mat_2} we have
\begin{align*}
&\det(\id-H-U_\alpha)\\
=&\det(\id-A)\det(\id-B-(1+\alpha)^2C^*(\id-A)^{-1}C)\\
\geq&\det(\id-A)\det\left[\frac{9-(1+\alpha)^2}{8}\cdot\left(\id-B-C^*(\id-A)^{-1}C\right)\right]\\
\geq&\left[\frac{9-(1+\alpha)^2}{8}\right]^8\det(\id-A)\det(\id-B-C^*(\id-A)^{-1}C)\\
=&\left[\frac{9-(1+\alpha)^2}{8}\right]^8\det(\id-H).
\end{align*}
Hence
\begin{align}\label{K_1}
\frac{\det(\id-H)}{\det(\id-H-U_\alpha)}\leq K_1(\alpha):=\left[\frac{9-(1+\alpha)^2}{8}\right]^{-8}.
\end{align}

\textbf{Finding an upper bound $K_2$ for $\det(H)^{4/15}/\det(\id-H)$:} Denote by $\nu_1\geq...\geq\nu_{16}\geq 0$ the eigenvalues of $H$. By the assumptions of \textbf{\hyperlink{case2}{Case 2}}, we have 
\begin{align}\label{7_ineq}
\nu_1\geq\mu_1\geq\mu-\mu_1=\sum_{j=2}^8\mu_j\geq\sum_{j=1}^7\nu_{16-j+1}.\end{align}
Define
$$F(x_1,...,x_{16})=\prod_{j=1}^{16}\frac{x_j^{4/15}}{1-x_j},\quad \sum_{j=1}^{16}x_j=1\mathrm{~and~}x_j\geq 0.$$
Then we have
\begin{align}\label{F(eigen)}F(\nu_1,...,\nu_{16})=\frac{\det(H)^{4/15}}{\det(\id-H)}.\end{align}
Define $f(t)=-\ln(1-t)+4\ln(t)/15$ with $t\in(0,1)$. Then $\ln F(x_1,...,x_{16})=\sum_{j=1}^{16}f(x_j)$. Direct computations yield
$$f'(t)=\frac{4}{15t}+\frac{1}{1-t},\quad f''(t)=-\frac{4}{15t^2}+\frac{1}{(1-t)^2}.$$
Therefore $f(t)$ is concave down on $\left(0,\sqrt{4/15}/\left(1+\sqrt{4/15}\right)\right)$ and concave up on $(\sqrt{4/15}/(1+\sqrt{4/15}),1)$. Since $\sqrt{4/15}/\left(1+\sqrt{4/15}\right)>1/3$ and $\nu_1+...+\nu_{16}=\tr(H)=1$, we have $\nu_j<\sqrt{4/15}/\left(1+\sqrt{4/15}\right)$ for any $3\leq j\leq 16$. Therefore we can use concavity arguments to find an upper bound for $F(\nu_1,...,\nu_{16})$.
\begin{lem}\label{F_Gmax}
Under the above assumptions on $F$ and $f$, there exist some $\nu_1'\geq...\geq\nu_{16}'$ such that
$$F(\nu_1,...,\nu_{16})\leq F(\nu_1',...,\nu_{16}').$$
Moreover, $(\nu_1',...,\nu_{16}')$ satsifies the following conditions 
\begin{enumerate}
\item[(1).] $\nu_2'=...=\nu_9'$ and $\nu_{10}'=...=\nu_{16}'$;
\item[(2).] $\nu_1'=\max\left\{\nu_1,\nu_1+\nu_2-\frac{\sqrt{4/15}}{1+\sqrt{4/15}}\right\}\geq\sum_{j=1}^7\nu_{16-j+1}'\geq\sum_{j=1}^7\nu_{16-j+1}$;
\item[(3).] If the first inequality in the second condition is strict, then $\nu_2'=...=\nu_{16}'$;
\item[(4).] If $\nu_2'>\nu_{16}$, then $\nu_1'=\sum_{j=1}^7\nu_{16-j+1}'$;
\item[(5).] $\sum_{j=1}^{16}\nu_j'=1$.
\end{enumerate}
\end{lem}
\begin{proof}[Proof of Lemma \ref{F_Gmax}]
The main idea behind this inequality is to apply a few adjustments to the original $\nu_j$'s so that the value of $F$ goes up. We first use the fact that $f(x)$ is concave up on $(\sqrt{4/15}/(1+\sqrt{4/15}),1)$ to obtain
$$f(\nu_1)+f(\nu_2)\leq f(\nu_1')+f(\nu_1+\nu_2-\nu_1'),$$
where $\nu_1'=\max\left\{\nu_1,\nu_1+\nu_2-\frac{\sqrt{4/15}}{1+\sqrt{4/15}}\right\}$ and $\nu_1+\nu_2-\nu_1'=\min\left\{\nu_2,\frac{\sqrt{4/15}}{1+\sqrt{4/15}}\right\}$. Notice that
$$\ln F(\nu_1,...\nu_{16})=\sum_{j=1}^{16}f(\nu_j).$$
Therefore
\begin{align}\label{step1_ineq}
F(\nu_1,...,\nu_{16})\leq F(\nu_1',\nu_1+\nu_2-\nu_1',\nu_3,...,\nu_{16}).\end{align}
Let $x'=(\nu_{10}+...+\nu_{16})/7\leq \sqrt{4/15}/(1+\sqrt{4/15})$ and $y'=[(\nu_1+\nu_2-\nu_1')+\nu_3...+\nu_9]/8\leq \sqrt{4/15}/(1+\sqrt{4/15})$. It is clear from the ordering of $\nu_j$ that $y'\geq x'$. Since $\sqrt{4/15}/(1+\sqrt{4/15})\geq (\nu_1+\nu_2-\nu_1')\geq \nu_3\geq...\geq\nu_{16}$, by the fact that $f(x)$ is concave down on $(0,\sqrt{4/15}/(1+\sqrt{4/15}))$, we have
$$f(\nu_1+\nu_2-\nu_1')+f(\nu_3)+...+f(\nu_9)\leq 8f(y')$$
and
$$f(\nu_{10})+...+f(\nu_{16})\leq 7f(x').$$
These two inequalities imply that
\begin{align}\label{step2_ineq}
F(\nu_1',\nu_1+\nu_2-\nu_1',\nu_3,...,\nu_{16})\leq F(\nu_1',y',...,y',x'...,x'),
\end{align}
where $y'$ shows up $8$ times and $x'$ shows up $7$ times.

\textbf{Case (1):} If $\nu_1'\geq 7(8y'+7x')/15$, then by the fact that $f(x)$ is concave down on $(0,\sqrt{4/15}/(1+\sqrt{4/15}))$, we have
$$8f(y')+7f(x')\leq 15f((8y'+7x')/15).$$
Choose $\nu_2'=...'\nu_{16}'=(8y'+7x')/15$. Then
\begin{align}\label{step3_ineq1}
F(\nu_1',y',...,y',x'...,x')\leq F(\nu_1',...,\nu_{16}').
\end{align}
Lemma \ref{F_Gmax} then follows from \eqref{step1_ineq}, \eqref{step2_ineq} and \eqref{step3_ineq1}.

\textbf{Case (2):} If $\nu_1'< 7(8y'+7x')/15$, we let $x=\nu_1'/7$ and $y=(1-2\nu_1)/8=(8y'+7x'-7x)/8>x$. Clearly $8y'+7x'=8y+7x$. Since $\nu_1'\geq\nu_1\geq\nu_{10}+...+\nu_{16}=7x'$, we have $y'\geq y>x\geq x'$. By the fact that $f(x)$ is concave down on $(0,\sqrt{4/15}/(1+\sqrt{4/15}))$, we have
$$8f(y')+7f(x')\leq 8f(y)+7f(x).$$
Choose $\nu_2'=...=\nu_9'=y$ and $\nu_{10}'=...=\nu_{16}'=x$, we have
\begin{align}\label{step3_ineq2}
F(\nu_1',y',...,y',x'...,x')\leq F(\nu_1',...,\nu_{16}').
\end{align}
Lemma \ref{F_Gmax} then follows from \eqref{step1_ineq}, \eqref{step2_ineq} and \eqref{step3_ineq2}.
\end{proof}
Back to finding an upper bound $K_2$ for $\det(H)^{4/15}/\det(\id-H)=F(\nu_1,...,\nu_{16})$. Let $\nu_1'\geq\nu_2'=...=\nu_9'=y\geq\nu_{10}'=...=\nu_{16}'=x$ be the same as in Lemma \ref{F_Gmax}. Then we have $\nu_1'\geq7x$ and $y\geq x$. Moreover, $\nu_1'>7x$ implies that $y=x$.

\textbf{Case 2a}: $\nu_1'=7x$. Define 
$$L(x)=\frac{(7x^8)^{4/15}}{(1-x)^7(1-7x)}\cdot\left[\frac{(14x/8)^{32/15}}{(1-14x/8)^8}\right]^{-1},~~0\leq x\leq \frac{1}{22}.$$
Then by \cite[B.4 Lemme]{BCG1} we have
$$F(7x,y,...,y,x,...,x)=L(x)F(14x/8,y,...,y,14x/8,...,14x/8)\leq L(x)F(\id/16),$$
where $y$ shows up $8$ times. Notice that
$$\left(\ln L(x)\right)'=\frac{7}{1-7x}+\frac{7}{1-x}-\frac{14}{(1-14x/8)}\geq 0.$$
Therefore by Lemma \ref{F_Gmax} and \eqref{F(eigen)} we have
\begin{align}\label{K_2}\det(H)^{4/15}/\det(\id-H)=F(\nu_1,...,\nu_{16})\leq F(\nu_1',...,\nu_{16}')\leq& L(1/22)F(1/16,...,1/16) \nonumber\\=&L(1/22)\cdot\frac{16^{16(1-4/15)}}{15^{16}}=:K_2.\end{align}

\textbf{Case 2b}: $\nu_1>7x$ and therefore $y=x$. Let 
$$R(x)=F(1-15x,x,...,x)=\frac{(1-15x)^{4/15}x^{4}}{15x(1-x)^{15}},~~0\leq x<\frac{1}{22}.$$
Then
$$(\ln R(x))'=-\frac{4}{1-15x}+\frac{3}{x}+\frac{15}{1-x}>0.$$
Therefore $R(x)$ is increasing and $F(1-15x,x,...,x)\leq F(7/22,1/22,...,1/22)\leq L(1/22)F(1/16,...,1/16)$. Hence this case reduces to Case 2a.

By \eqref{goal_of_step_3}, \eqref{K_1} and \eqref{K_2} it suffices to find an $\alpha$ satisfying \eqref{alpha_lower_bd} and
$$\left[\frac{9-(1+\alpha)^2}{8}\right]^{-8}\cdot L(1/22)\leq 1.$$
The above inequality is equivalent to
\begin{align}\label{alpha_upper_bd}
\alpha\leq-1+\sqrt{9-8L(1/22)^{1/8}}\in [0.266,0.267],
\end{align}
which gives a sufficient condition on $\alpha$ such that the second requirement in Lemma \ref{key_lemma} holds in \textbf{\hyperlink{case2}{Case 2}}.

Notice that the lower bound in \eqref{alpha_lower_bd} is smaller than $0.2$ and the smallest upper bound for $\alpha $ is greater than $0.26$. We can choose $\alpha=1/4$ and Lemma \ref{key_lemma} follows in \hyperlink{case2}{\textbf{Case 2}}. 
\end{proof}
\begin{rmk}
\begin{enumerate}
\item[(1).] In fact, the equality of the second inequality in Lemma \ref{key_lemma} can not be achieved in \hyperlink{case2}{\textbf{Case 2}} since \eqref{alpha_upper_bd} is strict when $\alpha=1/4$. To be specific, \eqref{K_1}, \eqref{K_2} and \eqref{alpha_upper_bd} imply that $0<K_1(1/4)K_2<1$. Therefore by \eqref{rewrite_new_part_1} and \eqref{goal_of_step_3} we have
\begin{align*}
\det(\id-H-U_\alpha)\geq\frac{15^{16}}{K_1(1/4)K_216^{16(1-4/15)}}\det(H)^{4/15}>\frac{15^{16}}{16^{16(1-4/15)}}\det(H)^{4/15}.
\end{align*}
Since the first inequality in Lemma \ref{key_lemma} holds in \hyperlink{case2}{\textbf{Case 2}} when $\alpha=1/4$, it follows from \eqref{new_trick} and the above inequality that
\begin{align}\label{strict_ineq}
\det(\id-H+(\widehat{H}-H))&\geq22^{16}\det((\id-H-U_\alpha)/15)^{15/22}\det((\widehat{H}-H+U_\alpha)/7)^{7/22} \nonumber\\
&\geq 22^{16}16^{-8}K_1(1/4)^{-15/22}K_2^{-15/22}\det(H)^{4/22+7/22} \nonumber\\
&>22^{16}16^{-8}\det(H)^{1/2}.
\end{align}
This shows that the equality in \eqref{simplification} (and equivalently Proposition \ref{ineq_prop}) cannot be achieved under the assumptions of \hyperlink{case2}{\textbf{Case 2}}. 
\item[(2).] The original proof of \cite{BCG1} actually proved a stronger inequality in the sense that the power of $\det(H)$ in the numerator can be reduced to $d/(n+d-2)$, where $d=1,2,4$ in the real, complex and quaternionic setting respectively and $n$ denotes the dimension of the symmetric space. The estimates in our proof are clearly not optimal since we can choose $\alpha$ from the interval $[0.2,0.26]$. Using the same method we can reduce the power of $\det(H)$ in the numerator to $-\epsilon+1/2$ for some small $\epsilon>0$, which is sufficient to prove entropy rigidity in the Cayley hyperbolic setting but weaker than the corresponding inequality in \cite{BCG1}. We suspect that in order to prove the stronger inequality as in \cite{BCG1}, we need some stronger estimates on $\det(\widehat{H}-H)$ or some new ideas.
\end{enumerate}
\end{rmk}

\subsection{Repairing the gap: proof of \eqref{sharp_det_est}}\label{s6}
Following the notations in previous subsections, we have the following analogue of \cite[B.5 Proposition]{BCG1}, which proves \eqref{sharp_det_est} in the Cayley hyperbolic setting.
\begin{cor}\label{key_cor}
Let $H$ and $\widehat{H}$ be as in Proposition \ref{ineq_prop}. Then there exists a constant $K$ such that
\begin{align}\label{det_error}
\frac{\left(\det{H}\right)^{1/2}}{\det\left(\id-2H+\widehat{H}\right)}\leq \left(\frac{\sqrt{16}}{16+8-2}\right)^{16}\left(1-K\sum_{j=1}^{16}\left(\nu_i-\frac{1}{16}\right)^2\right),
\end{align}
where $\nu_1,...,\nu_{16}$ denote the eigenvalues of $H$.
\end{cor}
\begin{proof}
Recall that in Proposition \ref{ineq_prop} we proved that under a suitable basis one can write
$$H=\matii{A&C}{C^*&B}\mathrm{~and~}\widehat{H}=\matii{\lambda\cdot\id&0}{0&\mu\cdot\id},$$
where $A, B$ are $8\times 8$ diagonal matrices with trace equal to $\lambda,\mu\geq 0$ respectively such that $\lambda+\mu=1$. Moreover,
$$A=\matiii{\lambda_1&&}{&\ddots&}{&&\lambda_8}\mathrm{~and~}B=\matiii{\mu_1&&}{&\ddots&}{&&\mu_8},$$
where $\lambda_1\geq...\geq\lambda_8\geq0$ and $\mu_1\geq...\geq\mu_8\geq0$ are eigenvalues of $A$ and $B$ respectively. Similar to the proof of Proposition \ref{ineq_prop}, we consider two separate cases.

\textbf{Case 1.} If $2\lambda_1\leq\lambda$ and $2\mu_1\leq\mu$, then \eqref{old_trick} and \eqref{good_case} suggest that it suffices to prove
$$\frac{\det\left(H\right)^{1/2-7/22}}{\det(I-H)^{15/22}}\leq \frac{16^8}{(15^{16})^{15/22}}\left(1-K\sum_{j=1}^{16}\left(\nu_i-\frac{1}{16}\right)^2\right)$$
for some constant $K$. This follows directly from the proof of \cite[Proposition B.5]{BCG1}.

\textbf{Case 2.} Under the assumptions of \hyperlink{case2}{\textbf{Case 2}} in the proof for Lemma \ref{key_lemma}, the left hand side of \eqref{simplification} has a strictly smaller upper bound comparing to its right hand side. This follows from \eqref{strict_ineq} in the remark of Lemma \ref{key_lemma}. If $2\lambda_1>\lambda$ or $2\mu_1>\mu$, then the corollary directly follows from the fact that $\sum_{j=1}^{16}\left(\nu_i-\frac{1}{16}\right)^2<16$.
\end{proof}
\begin{rmk}
Equivalently, Subsections \ref{s4}-\ref{s6} guarantees that all results in \cite[Appendice B]{BCG1} are correct in the Cayley hyperbolic setting.
\end{rmk}

\Addresses

\begin{thebibliography}{9}

\bibitem{Al}
Albuquerque, P.
\textit{Patterson-Sullivan theory in higher rank symmetric spaces. }
Geom. Funct. Anal. 9 (1999), no. 1, 1–28.

\bibitem{Baez}

Baez, J. C., \textit{The octonions}. Bull. Amer. Math. Soc. (N.S.) 39 (2002), no. 2, 145–205.

\bibitem{BCG1} Besson, G.; Courtois, G.; Gallot, S., \textit{Entropies et rigidités des espaces localement symétriques de courbure strictement négative}. (French) [Entropy and rigidity of locally symmetric spaces with strictly negative curvature] Geom. Funct. Anal. 5 (1995), no. 5, 731–799.

\bibitem{BCG2} Besson, G.; Courtois, G.; Gallot, S., \textit{Minimal entropy and Mostow's rigidity theorems}. Ergodic Theory Dynam. Systems 16 (1996), no. 4, 623–649.

\bibitem{BCS} Boland, J.; Connell, C.; Souto, J., \textit{Volume rigidity for finite volume manifolds}. Amer. J. Math. 127 (2005), no. 3, 535–550.

\bibitem{CF0}
Connell, C.; Farb, B., \textit{Minimal entropy rigidity for lattices in products of rank one symmetric spaces}. Comm. Anal. Geom. 11 (2003), no. 5, 1001–1026.


\bibitem{CF} Connell, C.; Farb, B.,\textit{ The degree theorem in higher rank.} J. Differential Geom. 65 (2003), no. 1, 19–59.
Erratum for "The degree theorem in higher rank''. J. Differential Geom. 105 (2017), no. 1, 21–32.

\bibitem{Mostow73}

Mostow, G. D., \textit{Strong rigidity of locally symmetric spaces.} Annals of Mathematics Studies, No. 78. Princeton University Press, Princeton, N.J.; University of Tokyo Press, Tokyo, 1973. v+195 pp.

\bibitem{Parker}

Parker, J. R., \textit{Hyperbolic spaces.} Jyväskylä Lectures in Mathematics 2 (2008).

\bibitem{Ruan1}

Ruan, Y., \textit{Filling volume minimality and boundary rigidity of metrics close to a negatively curved symmetric metric.} arXiv:2201.09175 [math.DG]

\bibitem{Springer00}

Springer, T. A., and Veldkamp, F. D., \textit{Octonions, Jordan algebras and exceptional groups}. Springer Monographs in Mathematics. Springer-Verlag, Berlin, 2000. viii+208 pp. ISBN: 3-540-66337-1.

\bibitem{Springer1}

Springer, T. A., and Veldkamp, F. D., \textit{Elliptic and hyperbolic octave
planes I, II, III}. Ibid, vol. 66, pp. 413-451 (1963).

\bibitem{Storm}
Storm, P. A.,
\textit{The minimal entropy conjecture for nonuniform rank one lattices. }
Geom. Funct. Anal. 16 (2006), no. 4, 959–980.

\bibitem{Yokota09}

Yokota, I., \textit{Exceptional Lie group}, arXiv:0902.0431 [math.DG].








\end{thebibliography}
\end{document}